\documentclass{amsart}
\usepackage[utf8]{inputenc}

\usepackage{changepage}
\usepackage{longtable}
\DeclareRobustCommand{\SkipTocEntry}[5]{}

\usepackage{pdflscape}
\usepackage{comment}

\usepackage{amsmath}
\usepackage{mathtools}
\usepackage{amssymb}

\usepackage{dirtytalk}
\usepackage{amsthm}
\usepackage{float}

\usepackage{wrapfig}

\usepackage{caption}

\usepackage{graphicx}

\usepackage{array}
\usepackage{tabularx}
\usepackage{longtable}

\usepackage[margin=1.05in]{geometry}

\usepackage[table,dvipsnames]{xcolor}
\definecolor{bytreebluevertex}{RGB}{128,128,255}
\definecolor{bytreeblueline}{RGB}{77,77,255}
\definecolor{bytreeyellowline}{RGB}{255,169,82}
\definecolor{bytreeyellowvertex}{RGB}{255,232,191}
    
\usepackage{hyperref}

 \usepackage[UKenglish]{babel}
\usepackage{graphicx}	
\usepackage{amsthm}
\usepackage{amsmath}
\usepackage{amssymb}
\usepackage{mathtools}
\usepackage{enumerate}
\usepackage{mathrsfs}
\usepackage{thmtools, thm-restate}
\usepackage[table]{xcolor}
\usepackage{tabularx, booktabs}
\usepackage{tikz}
\usepackage{pgfplots}

\usepackage{dsfont}
\usepackage[OT2,T1]{fontenc}

\DeclareSymbolFont{cyrletters}{OT2}{wncyr}{m}{n}
\DeclareMathSymbol{\Sha}{\mathalpha}{cyrletters}{"58}

\usepackage{tikz-cd} 
\usetikzlibrary {arrows.meta,bending,positioning,patterns,decorations}
\usepackage{pbox}
\usepackage{tkz-graph}
\usetikzlibrary{snakes,fit,positioning,calc,shapes}
\usepackage{graphicx}
\usepackage{cleveref,tikz-cd,pbox,wrapfig}
\definecolor{amethyst}{rgb}{0.6, 0.4, 0.8}
\definecolor{atomictangerine}{rgb}{1.0, 0.6, 0.4}
\definecolor{deeppeach}{rgb}{1.0, 0.8, 0.64}
\definecolor{eggshell}{rgb}{0.94, 0.92, 0.84}
\definecolor{lightapricot}{rgb}{0.99, 0.84, 0.69}
\definecolor{lemonchiffon}{rgb}{1.0, 0.98, 0.8}
\definecolor{roundabout}{rgb}{1.0, 0.91, 0.75}
\definecolor{atomictangerine}{rgb}{1.0, 0.6, 0.4}
\definecolor{ruby}{rgb}{0.88, 0.07, 0.37}
\definecolor{sapphire}{rgb}{0.03, 0.15, 0.4}

\def\rootsep{0.03}               % horizontal space between roots in a cluster picture
\def\clustersep{0.06}            % horizontal and vertical space between parent & child cluster
\def\cnamescale{0.4}             % cluster names font size 
\def\cdepthscale{0.4}            % cluster depths and signs font size 
\def\cltopskip{1pt}              % space above a cluster picture
\def\clbottomskip{1pt}           % space below a cluster picture

\def\rootscale{0.5}   \def\rootcolor{gray}
\def\rootscaleA{0.7}  \def\rootcolorA{yellow}
\def\rootscaleB{0.5}  \def\rootcolorB{green}
\def\rootscaleC{0.4}  \def\rootcolorC{sapphire}
\def\rootscaleD{0.45}  \def\rootcolorD{ruby}
\tikzset{
  clA/.style = {very thick,black},
  clB/.style = {thick,purple}
}

\def\graphdslabelscale{0.6}

\def\GraphScale{0.6}

\tikzset{
  root/.style = {circle,scale=\rootscale,fill=\rootcolor},
    rc/.style 2 args = {right=#1*1.5*\clustersep of {#2.east|-first},root}, rr/.style = {right=\rootsep of {#1.east|-first},root},
  roott/.style = {circle,inner sep=-2pt,minimum size=5pt,black,font=\ttfamily\footnotesize},
    rct/.style 2 args = {right=#1*1.5*\clustersep of {#2.east|-first},roott}, rrt/.style = {right=\rootsep of {#1.east|-first},roott},
  rootA/.style = {circle,scale=\rootscaleA,ball color=\rootcolorA},
    rcA/.style 2 args = {right=#1*1.5*\clustersep of {#2.east|-first},rootA}, rrA/.style = {right=\rootsep of {#1.east|-first},rootA},
  rootB/.style = {circle,scale=\rootscaleB,ball color=\rootcolorB},
    rcB/.style 2 args = {right=#1*1.5*\clustersep of {#2.east|-first},rootB}, rrB/.style = {right=\rootsep of {#1.east|-first},rootB},
  rootC/.style = {diamond,scale=\rootscaleC,ball color=\rootcolorC},
    rcC/.style 2 args = {right=#1*1.5*\clustersep of {#2.east|-first},rootC}, rrC/.style = {right=\rootsep of {#1.east|-first},rootC},
  rootD/.style = {circle,scale=\rootscaleD,ball color=\rootcolorD},
    rcD/.style 2 args = {right=#1*1.5*\clustersep of {#2.east|-first},rootD}, rrD/.style = {right=\rootsep of {#1.east|-first},rootD},
  cluster/.style = {draw=black!90,thick,rounded corners,inner sep=22*\clustersep,outer xsep=22*\clustersep,fit=#1},
  clabel/.style  = {anchor=west,scale=\cdepthscale,black,inner sep=0,outer xsep=1,outer ysep=0},
  clabelL/.style = {above right=-\clustersep of #1t.north east,clabel},
  clabelD/.style = {below right=-\clustersep of #1t.south east,clabel},
  clouter/.style = {inner sep=0,outer sep=0,fit=#1}
}

\def\Cluster #1 = #2;{\node[cluster=#2] (#1) {};}
\def\ClusterL #1[#2] = #3;{
  \node[cluster=#3] (#1t) {}; \node[clabelL=#1] (#1l) {$#2$}; \node[clouter=(#1t)(#1l)] (#1) {};}
\def\ClusterD #1[#2] = #3;{
  \node[cluster=#3] (#1t) {}; \node[clabelD=#1] (#1d) {$#2$}; \node[clouter=(#1t)(#1d)] (#1) {};}
\def\ClusterLD #1[#2][#3] = #4;{
  \node[cluster=#4] (#1t) {}; \node[clabelL=#1] (#1l) {$#2$}; 
  \node[clabelD=#1] (#1d) {$#3$}; \node[clouter=(#1t)(#1l)(#1d)] (#1) {};}
\def\ClusterLDName #1[#2][#3][#4] = #5;{
  \node[cluster=#5] (#1t) {}; \node[clabelL=#1] (#1l) {$#2$}; 
  \node[clabelD=#1] (#1d) {$#3$}; 
  \node[scale=\cnamescale,above=\clustersep/3 of #1t,inner sep=0, outer sep=0] (#1n) {$#4$}; 
  \node[clouter=(#1l)(#1d)(#1t)] (#1) {};}

\newcommand{\Root}[4][]{
  \ifx\relax#2\relax\node[rr#1=#3] (#4) {};\else\node[rc#1={#2}{#3}] (#4) {};\fi}
\newcommand{\RootT}[5][]{
  \ifx\relax#2\relax\node[rrt#1=#3] (#4) {#5};\else\node[rct#1={#2}{#3}] (#4) {#5};\fi}

\def\frob(#1)(#2){\path[draw,thick,shorten <=-22*\clustersep,shorten >=-22*\clustersep](#1.east)--(#2.west|-#1){};}

\usepackage{pbox}
\def\pb#1{\pbox[c]{\textwidth}{\hfil #1\hfil}}

\long\def\clusterpicture#1\endclusterpicture{\pb{\vbox to \cltopskip{\vfill}\\%
  \begin{tikzpicture}\node[coordinate] (first) {};#1\end{tikzpicture}\\[-11pt]\vbox to \clbottomskip{\vfill}}}   
\long\def\clusterpictureopt#1#2\endclusterpicture{\pb{\vbox to \cltopskip{\vfill}\\%
  \begin{tikzpicture}[#1]\node[coordinate] (first) {};#2\end{tikzpicture}\\[-11pt]\vbox to \clbottomskip{\vfill}}}

\def\pb#1{\pbox[c]{\textwidth}{\hfil #1\hfil}}

\def\GraphVertices{\SetVertexNormal[Shape=circle, FillColor=blue!50, LineColor=blue!50, LineWidth=0.8pt]
  \tikzset{VertexStyle/.append style = {inner sep=0.5pt,minimum size=0.3em,font = \tiny\bfseries}}}

\def\BlueEdges{  \SetUpEdge[lw=0.8pt,color=blue!70]
   \tikzset{EdgeStyle/.append style = {shorten <=0.5pt,shorten >=0.5pt}}}

\def\LoopW(#1){
  \path[draw,-,thick,color=blue!70] (#1) edge[out=155,in=90] ($(#1)-(1.3,0)$);
  \path[draw,-,thick,color=blue!70] (#1) edge[out=210,in=270] ($(#1)-(1.3,0)$);
}
\def\LoopE(#1){
  \path[draw,-,thick,color=blue!70] (#1) edge[out=25,in=90] ($(#1)+(1.3,0)$);
  \path[draw,-,thick,color=blue!70] (#1) edge[out=-25,in=270] ($(#1)+(1.3,0)$);
}
\def\LoopS(#1){
  \path[draw,-,thick,color=blue!70] (#1) edge[out=115,in=180] ($(#1)+(0,1.2)$);
  \path[draw,-,thick,color=blue!70] (#1) edge[out=65,in=0] ($(#1)+(0,1.2)$);
}
\def\LoopN(#1){
  \path[draw,-,thick,color=blue!70] (#1) edge[out=-115,in=180] ($(#1)-(0,1.2)$);
  \path[draw,-,thick,color=blue!70] (#1) edge[out=-65,in=0] ($(#1)-(0,1.2)$);
}
\def\EdgeW(#1){
  \path[draw,-,thick,color=blue!70] (#1+) edge[out=180,in=90] ($(#1+)-(1.3,0.3)$);
  \path[draw,-,thick,color=blue!70] (#1-) edge[out=180,in=270] ($(#1+)-(1.3,0.3)$);
}
\def\EdgeE(#1){
  \path[draw,-,thick,color=blue!70] (#1+) edge[out=0,in=90] ($(#1-)+(1.3,0.3)$);
  \path[draw,-,thick,color=blue!70] (#1-) edge[out=0,in=270] ($(#1-)+(1.3,0.3)$);
}
\def\EdgeS(#1){
  \path[draw,-,thick,color=blue!70] (#1+) edge[out=90,in=0] ($(#1-)+(0.3,1.3)$);
  \path[draw,-,thick,color=blue!70] (#1-) edge[out=90,in=180] ($(#1-)+(0.3,1.3)$);
}
\def\EdgeN(#1){
  \path[draw,-,thick,color=blue!70] (#1+) edge[out=270,in=0] ($(#1+)-(0.3,1.3)$);
  \path[draw,-,thick,color=blue!70] (#1-) edge[out=270,in=180] ($(#1+)-(0.3,1.3)$);
}
\def\GCircle(#1,#2)(#3,#4){
  \path(#1,#2) node[coordinate] (1) {};
  \path(#3,#4) node[coordinate] (2) {};
  \path[draw,-,thick,color=blue!70] (1) edge[out=90,in=90] (2);
  \path[draw,-,thick,color=blue!70] (2) edge[out=270,in=270] (1);
}

\def\EdgeSign(#1)(#2)#3(#4)#5{
  \node at ($(#1)!#3!(#2) + (#4)$) [color=black, scale=\graphdslabelscale] {$\scriptstyle #5$};
}

\def\GraphEdgeSignDist{0.55}

\def\GraphEdgeSignS(#1)(#2)#3#4{\EdgeSign(#1)(#2)#3(0,-\GraphEdgeSignDist){#4}}

\def\VSwap#1#2#3#4{\path[draw](#1) edge[<->,#3,shorten >=#4pt,shorten <=#4pt] (#2){};}
\def\VArr#1#2#3#4{\path[draw](#1) edge[->,#3,shorten >=#4pt,shorten <=#4pt] (#2){};}

\def\ESwapOfs#1#2#3#4#5#6#7#8{\VSwap{$(#1)!0.5!(#2) + (#6)$}{$(#3)!0.5!(#4) + (#7)$}{#5}{#8}}

\def\EArrOfs#1#2#3#4#5#6#7#8{\VArr{$(#1)!0.5!(#2) + (#6)$}{$(#3)!0.5!(#4) + (#7)$}{#5}{#8}}

\def\tgrGB{\raise-7pt\hbox{\begin{tikzpicture}[scale=\GraphScale]
  \GraphVertices
  \Vertex[x=1.50,y=0.000,L=1]{1};
  \coordinate (2) at (0.000,0.000);
  \BlueEdges
  \LoopW(1)
\GraphEdgeSignS(1)(2){0.5}{n}\end{tikzpicture}}}

\def\tgrGBex{\raise-7pt\hbox{\begin{tikzpicture}[scale=\GraphScale]
  \GraphVertices
  \Vertex[x=1.50,y=0.000,L=1]{1};
  \coordinate (2) at (0.000,0.000);
  \BlueEdges
  \LoopW(1)
\GraphEdgeSignS(1)(2){0.5}{1}\end{tikzpicture}}}

\def\tgrGC{\raise-7pt\hbox{\begin{tikzpicture}[scale=\GraphScale]
  \GraphVertices
  \Vertex[x=1.50,y=0.000,L=1]{1};
  \coordinate (2) at (0.000,0.000);
  \BlueEdges
  \LoopW(1)
\GraphEdgeSignS(1)(2){0.5}{n}\ESwapOfs1212{}{0,-0.25}{0,0.25}{0.5}\end{tikzpicture}}}

\def\tgrGD{\raise-7pt\hbox{\begin{tikzpicture}[scale=\GraphScale]
  \GraphVertices
  \Vertex[x=1.50,y=0.000,L=\relax]{1};
  \coordinate (2) at (3.00,0.000);
  \coordinate (3) at (0.000,0.000);
  \BlueEdges
  \LoopE(1)
  \LoopW(1)
\GraphEdgeSignS(1)(3){0.5}{n}\GraphEdgeSignS(1)(2){0.5}{n}\end{tikzpicture}}}

\def\tgrGE{\raise-7pt\hbox{\begin{tikzpicture}[scale=\GraphScale]
  \GraphVertices
  \Vertex[x=1.50,y=0.000,L=\relax]{1};
  \coordinate (2) at (3.00,0.000);
  \coordinate (3) at (0.000,0.000);
  \BlueEdges
  \LoopE(1)
  \LoopW(1)
\GraphEdgeSignS(1)(3){0.5}{n}\GraphEdgeSignS(1)(2){0.5}{n}\ESwapOfs1212{}{0,-0.25}{0,0.25}{0.5}\end{tikzpicture}}}

\def\tgrGF{\raise-7pt\hbox{\begin{tikzpicture}[scale=\GraphScale]
  \GraphVertices
  \Vertex[x=1.50,y=0.000,L=\relax]{1};
  \coordinate (2) at (3.00,0.000);
  \coordinate (3) at (0.000,0.000);
  \BlueEdges
  \LoopE(1)
  \LoopW(1)
\GraphEdgeSignS(1)(3){0.5}{n}\GraphEdgeSignS(1)(2){0.5}{n}\ESwapOfs1313{}{0,-0.25}{0,0.25}{0.5}\ESwapOfs1212{}{0,-0.25}{0,0.25}{0.5}\end{tikzpicture}}}

\def\tgrGG{\raise-7pt\hbox{\begin{tikzpicture}[scale=\GraphScale]
  \GraphVertices
  \Vertex[x=1.50,y=0.000,L=\relax]{1};
  \coordinate (2) at (3.00,0.000);
  \coordinate (3) at (0.000,0.000);
  \BlueEdges
  \LoopE(1)
  \LoopW(1)
\GraphEdgeSignS(1)(3){0.5}{n}\GraphEdgeSignS(1)(2){0.5}{n}\ESwapOfs1312{in=160,out=20}{0.2,0.3}{-0.2,0.3}{0.5}\end{tikzpicture}}}

\def\tgrGH{\raise-7pt\hbox{\begin{tikzpicture}[scale=\GraphScale]
  \GraphVertices
  \Vertex[x=1.50,y=0.000,L=\relax]{1};
  \coordinate (2) at (3.00,0.000);
  \coordinate (3) at (0.000,0.000);
  \BlueEdges
  \LoopE(1)
  \LoopW(1)
\GraphEdgeSignS(1)(3){0.5}{n}\GraphEdgeSignS(1)(2){0.5}{n}\EArrOfs1312{in=150,out=30}{0.1,0.29}{0,0.35}{0.5}\EArrOfs1213{in=-60,out=-60}{0,0.2}{0.3,-0.25}{0.5}\end{tikzpicture}}}

\def\tgrGA{\raise-3pt\hbox{\begin{tikzpicture}[scale=\GraphScale]
  \GraphVertices
  \Vertex[x=0.000,y=0.000,L=2]{1};
  \BlueEdges
\end{tikzpicture}}}

\newtheorem{theorem}{Theorem}[section]
\newtheorem{corollary}[theorem]{Corollary}
\newtheorem{lemma}[theorem]{Lemma}

\theoremstyle{definition}
\newtheorem{definition}[theorem]{Definition}

\newtheorem{example}[theorem]{Example}

\newtheorem{algorithm}[theorem]{Algorithm}
\newtheorem{fact}[theorem]{Fact}
\newtheorem{notation}[theorem]{Notation}

\newtheorem{remark}[theorem]{Remark}

\usepackage{dirtytalk}
\usepackage{enumerate}

\usepackage{mathrsfs}

\usepackage{dsfont}

\usepackage{graphicx}

\usepackage{tikz}
\usetikzlibrary{shapes,positioning}
\usetikzlibrary{arrows,chains,matrix,positioning,scopes}

\makeatletter
\tikzset{join/.code=\tikzset{after node path={%
\ifx\tikzchainprevious\pgfutil@empty\else(\tikzchainprevious)%
edge[every join]#1(\tikzchaincurrent)\fi}}}
\makeatother

\tikzset{>=stealth',every on chain/.append style={join},
         every join/.style={->}}

\mathtoolsset{showonlyrefs}

\title{Recovering the cluster picture of a polynomial over a discretely valued field}
\author{Lilybelle Cowland Kellock}
\address{University College London, London WC1H 0AY, UK}
\email{lilybelle.kellock.20@ucl.ac.uk}

\begin{document}

\begin{abstract}
For $f(x)$ a separable polynomial of degree $d$ over a discretely valued field $K$, we describe how the cluster picture of $f(x)$ over $K$, in other words the set of tuples $\{(\textup{ord}(x_i-x_j),i,j) : 1\leq i< j \leq d \}$ where $x_1,\dots,x_d$ are the roots of $f(x)$,  can be recovered without knowing the roots of $f(x)$ over $\Bar{K}$. We construct an explicit list of polynomials $g_d^{(1)},\dots,g_d^{(t_d)}\in\mathbb{Z}[A_0,\dots,A_{d-1}]$ such that the valuations $\textup{ord}(g_{d}^{(i)}(a_0,\dots,a_{d-1}))$ for $i=1,\dots,t_d$ uniquely determine this set of distances for the polynomial $f(x)=c_f(x^d+a_{d-1}x^{d-1}+\dots+a_0)$, and we describe the process by which they do so. We use this to deduce that if $C:y^2=f(x)$ is a hyperelliptic curve over a local field $K$, this list of valuations of polynomials in the coefficients of $f(x)$ uniquely determines the dual graph of the special fibre of the minimal strict normal crossings model of $C/K^{\textup{unr}}$, the inertia action on the Tate module and the conductor exponent. This provides a hyperelliptic curves analogue to a corollary of Tate's algorithm, that in residue characteristic $p\geq 5$ the dual graph of special fibre of the the minimal regular model of an elliptic curve $E/K^{\textup{unr}}$ is uniquely determined by the valuation of $j_E$ and $\Delta_E$.
\end{abstract}

\maketitle

\setcounter{tocdepth}{1}
\tableofcontents

\section{Introduction}\label{intro}
Let $f(x)$ be a separable polynomial of degree $d$ over a discretely valued field $K$. In this paper, we address the question of how the set of tuples $\{(\textup{ord}(x_i-x_j),i,j) : 1\leq i< j \leq d\}$ up to reordering of the roots $x_1,\dots, x_d$ of $f(x)$, also known as the \textit{cluster picture}, can be recovered from valuations of polynomials in the coefficients of $f(x)$. The results of this paper mean that the configuration of the distances between the roots can be recovered without having to find the roots of $f(x)$ over $\Bar{K}$, which could be defined over large extensions if $d$ is large. The main result we prove is the following theorem, which states that the configuration of the distances between the roots of $f(x)$ can be recovered from a finite list of polynomials in the coefficients of $f(x)$, with this list depending only on the degree of $f(x)$. We describe this list explicitly in Theorem \ref{introthmfromg} below. Throughout, we use $\textup{ord}$ to denote the valuation with respect to a uniformiser of $K$. 

\begin{theorem}[=Theorem \ref{polythm}]\label{top}
There exists a finite and explicit list of polynomials $g_{d}^{(1)},\dots,g_{d}^{(t_d)}\in\mathbb{Z}[A_0,\dots,A_{d-1}]$ for which, if $f(x)=c_f(x^d+a_{d-1}x^{d-1}+\cdots+a_0)$ is a separable polynomial of degree $d$ over a discretely valued field $K$, $\textup{ord}(g_{d}^{(i)}(a_0,\dots,a_{d-1}))$ for $i=1,\dots,t_d$ uniquely determines the set of tuples
\begin{equation}
\{(\textup{ord}(x_i-x_j),i,j) : 1\leq i< j \leq d \},
\end{equation}
up to reordering of the roots $x_1,\dots, x_d$ of $f(x)$.
\end{theorem}

Knowing the configuration of the distances between the roots of a polynomial $f(x)$ over a discretely valued field is of significant importance to the study of elliptic and hyperelliptic curves. For instance, if $E:y^2=x^3+ax+b$ is an elliptic curve over a local field $K$ of residue characteristic $\geq 5$, there are two possibilities for the configuration of the roots of the cubic, and this tells us the reduction type of the curve. In some labelling of the roots $x_1$, $x_2$ and $x_3$, either 
\begin{enumerate}[(i)]
\item  $\textup{ord}(x_1-x_2)=\textup{ord}(x_1-x_3)=\textup{ord}(x_2-x_3)=d$ for some $d\in\mathbb{Q}$, or
\item  $\textup{ord}(x_1-x_2)=d_1$ and $\textup{ord}(x_1-x_3)=\textup{ord}(x_2-x_3)=d_2$ for some $d_1,d_2\in\mathbb{Q}$ with $d_1>d_2$,
\end{enumerate}
and $E/K$ has potentially good reduction if and only if $E/K$ has root configuration $(i)$. Further to this, we can read off the Kodaira type of the curve from the root configuration using Tate's algorithm (see, for example, \cite{omrisarah} Example 1.13). More generally, for a hyperelliptic curve given by a Weierstrass equation $C:y^2=f(x)$ over a discretely valued field $K$, extensive work has been undertaken on recovering important arithmetic information such as reduction types from the configuration of the differences of roots, the methodology for which was introduced in \cite{m2d2}. The goal of this paper is to provide a method for recovering the configuration of the differences of roots of a polynomial $f(x)$ from polynomials in the coefficients of $f(x)$, thus giving an analogue to a corollary of Tate's algorithm in the setting of hyperelliptic curves, that for an elliptic curve $E$ over a local field $K$ of residue characteristic $\geq 5$, one can obtain the Kodaira type of $E/K$ from the coefficients of a Weierstrass equation for $E$ using the valuation of $j_E$ and $\Delta_E$. We state the results of this paper pertaining to hyperelliptic curves in \S\ref{hyperellipticapplications}. 

In the following theorem we explicitly describe the polynomials from Theorem \ref{top} that recover the configuration of the distances between the roots. The polynomials are described using rational functions in the roots of $f(x)$ that are associated to weighted graphs on $\deg(f)$ vertices called \textit{auxiliary graphs} (see Definition \ref{boldg}). The rational functions in the roots of $f(x)$ are defined in terms of differences of roots so that their valuations can be related to the distances between the roots (see \S\ref{valuationssection}). They are quotients of polynomials that are symmetric in the roots of $f(x)$ so they can be written in terms of the coefficients of $f(x)$, and thus the roots of $f(x)$ over $\Bar{K}$ do not need to be known a priori in order to evaluate them.

\begin{theorem}[See Theorem \ref{firstalgorithm}]\label{introthmfromg}
Let $f(x)$ be a separable polynomial of degree $d$ over a discretely valued field $K$. The valuations $\textup{ord}(J_{G,f})$ for every $G\in \mathbf{G}_d$ (see Definitions \ref{invariantsintro} and \ref{boldg}) uniquely determine the set of tuples
\begin{equation}
\{(\textup{ord}(x_i-x_j),i,j) : 1\leq i<j \leq d\},
\end{equation}
up to reordering of the roots $x_1,\dots, x_d$ of $f(x)$.
\end{theorem}

Theorem \ref{introthmfromg} is a simplified version of Theorem \ref{introthm} below, which explicitly describes how to recover the set of tuples from the valuations.

\begin{definition}\label{invariantsintro}\label{invariants}
Let $G=(V,E,w)$ be a weighted graph, where $w:E\rightarrow \mathbb{Z}_{\geq 0}$. Fix a labelling of the vertices $V=\{v_1,\dots,v_d\}$ of $G$ corresponding to the variables $X_1,\dots,X_d$, thus considering $G$ as a labelled graph. There is a natural action of $\sigma\in S_d$ on $G$ via the action on the vertices. This is given explicitly by letting
\begin{equation}
    \sigma(G)=(V,\sigma(E),\sigma(w)),
\end{equation}
where $\sigma(E)=\{v_{\sigma(i)}v_{\sigma(j)}:v_iv_j\in E\}$ and $\sigma(w)(v_{\sigma(i)}v_{\sigma(j)})=w(v_iv_j)$. Let $\sigma\in S_d/\text{Stab}_{S_d}(G)$ under this action. We define
\begin{equation}
    S_{G}^{\sigma}(X_1,\dots,X_d)=\frac{1}{\displaystyle \prod_{v_iv_j\in \sigma(E)}(X_i-X_j)^{2\sigma(w)(v_iv_j)}}.
\end{equation}
To $G$, associate the rational function $J_{G}=\sum_{\sigma} S_{G}^{\sigma}(X_1,\dots,X_d)$, where the sum is taken over all $\sigma\in S_d/\text{Stab}_{S_d}(G)$. For a separable polynomial with a fixed labelling of the roots $f(x)=(x-x_1)\cdots(x-x_d)$, write $S_{G,f}^\sigma$ and $J_{G,f}$ for $S_G^\sigma(x_1,\dots,x_d)$ and $J_G(x_1,\dots,x_d)$ respectively, and note that $J_{G}$ and $J_{G,f}$ do not depend on the labelling of the vertices of $G$ or the roots of $f(x)$.
\end{definition}

\begin{example}\label{invgraphs}
For the complete graph on $d$ vertices $K_d$ where each edge has weight $1$, we have $J_{K_d}=1/\Delta$, where $\Delta=\prod_{i<j}(X_i-X_j)^2$. In general, the rational functions will all look like a symmetric polynomial in $X_1,\dots,X_d$ divided by a power of $\Delta$. For example, for the graph $K$ on $3$ vertices below its associated rational function is as follows, with $\Delta=(X_1-X_2)^2(X_1-X_3)^2(X_2-X_3)^2$.

\begin{minipage}{.1\textwidth}
\begin{figure}[H]
		\begin{center}
			
			  \begin{tikzpicture}
				[scale=0.5,auto=left,every node/.style={circle,fill=black!20,scale=0.7}]
				\node (n1) at (0,0)  {};
				\node (n2) at (2,0)  {};
				\node[label={[label distance=0.15cm, scale=1.25]180:$K$}] (n3) at (1,2)  {};
				
				\draw (n1)-- (n2)  node [text=black,pos=0.5, above,fill=none] {$2$};
				\draw (n1) -- (n3) node [text=black,pos=0.5, left,fill=none] {$1$};
				\draw (n2) -- (n3) node [text=black,pos=0.5, right,fill=none] {$1$};
			\end{tikzpicture}
		\end{center}
\end{figure}
\end{minipage}\hspace{20pt}\begin{minipage}{.5\textwidth}
		\begin{equation}
		   J_{K}=\frac{(X_3-X_2)^2(X_3-X_1)^2+(X_3-X_2)^2(X_2-X_1)^2+(X_1-X_2)^2(X_3-X_1)^2}{\Delta^2}
		\end{equation}
		\end{minipage}
\end{example}

\begin{definition}\label{boldg}
Define $\mathbf{G}_d$ to be the set of all weighted graphs on $d$ vertices $G=(V,E,w)$, for which:
    \begin{enumerate}[(a)]
        \item $w:E\twoheadrightarrow \{1,\dots,n\}$ for some $n\in\mathbb{Z}^+$; 
        \item If $1\leq k\leq n$ and all edges of weight $\leq k$ are removed, the remaining graph is a disjoint union of complete graphs. Equivalently, allocating the edges not in $E$ weight $0$, for $v_1,v_2,v_3\in V$, $w(v_1v_2)\geq \min(w(v_2v_3),w(v_1v_3))$.
    \end{enumerate}
We call graphs in $\mathbf{G}_d$ \textit{auxiliary graphs} on $d$ vertices.
\end{definition}

In Example \ref{ellipticcurvesintro} below, we write down two of the rational functions associated to graphs in $\mathbf{G}_3$, and use these to demonstrate how the rational functions from Theorem \ref{introthmfromg} can be used to recover the configuration of the distances between the roots of a cubic. The methods of this paper generalise the phenomenon outlined in this example to recovering more complicated root configurations for higher degree polynomials.

\begin{example}\label{ellipticcurvesintro}
The graphs in $\mathbf{G}_3$ are:
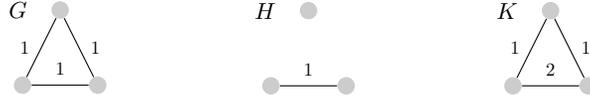
\begin{center}
\begin{minipage}{0.2\textwidth}
\begin{center}
    \begin{tikzpicture}
				[scale=0.5,auto=left,every node/.style={circle,fill=black!20,scale=0.7}]
				\node (n1) at (0,0)  {};
				\node (n2) at (2,0)  {};
				\node[label={[label distance=0.15cm, scale=1.25]180:$G$}] (n3) at (1,2)  {};
				
				\draw (n1) -- (n2) node [text=black,pos=0.5, above,fill=none] {$1$};
				\draw (n1) -- (n3)node [text=black,pos=0.5, left,fill=none] {$1$};
				\draw (n2) -- (n3)node [text=black,pos=0.5, right,fill=none] {$1$};
			\end{tikzpicture}
\end{center}
\end{minipage}\begin{minipage}{0.2\textwidth}
\begin{center}
\begin{tikzpicture}
				[scale=0.5,auto=left,every node/.style={circle,fill=black!20,scale=0.7}]
				\node (n1) at (0,0)  {};
				\node (n2) at (2,0)  {};
				\node[label={[label distance=0.15cm, scale=1.25]180:$H$}] (n3) at (1,2)  {};

				\draw (n1)-- (n2) node [text=black,pos=0.5, above,fill=none] {$1$};
			\end{tikzpicture}
\end{center}
\end{minipage}\begin{minipage}{0.2\textwidth}
\begin{center}
    \begin{tikzpicture}
				[scale=0.5,auto=left,every node/.style={circle,fill=black!20,scale=0.7}]
				\node (n1) at (0,0)  {};
				\node (n2) at (2,0)  {};
				\node[label={[label distance=0.15cm, scale=1.25]180:$K$}] (n3) at (1,2)  {};
				
				\draw (n1)-- (n2)  node [text=black,pos=0.5, above,fill=none] {$2$};
				\draw (n1) -- (n3) node [text=black,pos=0.5, left,fill=none] {$1$};
				\draw (n2) -- (n3) node [text=black,pos=0.5, right,fill=none] {$1$};
			\end{tikzpicture}
\end{center}
\end{minipage}
\end{center}
If $E:y^2=x^3+ax+b=f(x)$ is an elliptic curve over a discretely valued field, there are two possibilities for the configuration of the roots. In some labelling of the roots $x_1$, $x_2$ and $x_3$ of $f(x)$, either
\begin{enumerate}[(i)]
\item  $\textup{ord}(x_1-x_2)=\textup{ord}(x_1-x_3)=\textup{ord}(x_2-x_3)=d$ for some $d\in\mathbb{Q}$, or
\item  $\textup{ord}(x_1-x_2)=d_1$ and $\textup{ord}(x_1-x_3)=\textup{ord}(x_2-x_3)=d_2$ for some $d_1,d_2\in\mathbb{Q}$ with $d_1>d_2$. 
\end{enumerate}
The following table gives the rational function associated to the graphs with weight $1$ in $\mathbf{G}_3$ when evaluated on the roots of $f(x)$, and the valuation of these rational functions when evaluated on elliptic curves with root configurations $(i)$ and $(ii)$. We omit $J_{K,f}$ since $J_{K,f}=J_{G,f}\cdot J_{H,f}$ so its valuation does not give us any further information. 

\setlength\extrarowheight{2pt}
\begin{center}
{\setlength{\LTleft}{0.2in}
 \setlength{\LTright}{0in}
\begin{longtable}{ |>{\centering\arraybackslash} m{1.7cm} >{\centering\arraybackslash} m{5.5cm} | >{\centering\arraybackslash} m{3.1cm} | >{\centering\arraybackslash} m{3.4cm}|}

\cline{3-4} 

\multicolumn{2}{c}{} & \multicolumn{2}{|c|}{Valuation of $J_{-,f}$ in case $(i)$ and $(ii)$} \\
\hline 

Graph & Rational function $J_{-,f}$ & \multicolumn{1}{>{\raggedleft\arraybackslash} m{3.1cm}|}{$(i)$ $\textup{ord}(x_1-x_2)=d$ $\textup{ord}(x_1-x_3)=d$ $\textup{ord}(x_2-x_3)=d$} & \multicolumn{1}{>{\raggedleft\arraybackslash} m{3.4cm}|}{$(ii)$ $\textup{ord}(x_1-x_2)=d_1$ $\textup{ord}(x_1-x_3)=d_2$ $\textup{ord}(x_2-x_3)=d_2$} \\

\hline

 \begin{tikzpicture}
				[scale=0.5,auto=left,every node/.style={circle,fill=black!20,scale=0.7}]
     \node[fill=white] (n5) at (1,2.5) {};
				\node (n1) at (0,0)  {};
				\node (n2) at (2,0)  {};
    \node[fill=white] (n6) at (1,-0.5) {};
				\node[label={[label distance=0.15cm, scale=1.25]180:$G$}]  (n3) at (1,2)  {};
				
				\draw (n1) -- (n2) node [text=black,pos=0.5, above,fill=none] {$1$};
				\draw (n1) -- (n3)node [text=black,pos=0.5, left,fill=none] {$1$};
				\draw (n2) -- (n3)node [text=black,pos=0.5, right,fill=none] {$1$};
			\end{tikzpicture} & $\frac{1}{(x_1-x_2)^2(x_1-x_3)^2(x_2-x_3)^2}$ & $-6d$ & $-(4d_2+2d_1)$ \\

 \begin{tikzpicture}
				[scale=0.5,auto=left,every node/.style={circle,fill=black!20,scale=0.7}]
    \node[fill=white] (n5) at (1,2.5) {}; 
				\node (n1) at (0,0)  {};
				\node (n2) at (2,0)  {};
    \node[fill=white] (n6) at (1,-0.5) {};
				\node[label={[label distance=0.15cm, scale=1.25]180:$H$}] (n3) at (1,2)  {};
				
				\draw (n1)-- (n2) node [text=black,pos=0.5, above,fill=none] {$1$};
			\end{tikzpicture}  & 
   $\frac{1}{(x_1-x_2)^2} +\frac{1}{(x_1-x_3)^2} +\frac{1}{(x_2-x_3)^2}$ & $\geq -2d$ & $-2d_1$ \\
\hline 
\end{longtable}}
\end{center}
We can use these rational functions $J_{G,f}$ and $J_{H,f}$ to distinguish between case $(i)$ and $(ii)$, since, using the fact that $d_1>d_2$, the cubic has root configuration $(ii)$ if and only if $\frac{1}{6}\textup{ord}(J_{G,f})> \frac{1}{2}\textup{ord}(J_{H,f})$. If we are in case $(i)$ then $d=-\frac{1}{6}\textup{ord}(J_{G,f})$, and if we are in case $(ii)$ then $d_1=-\frac{1}{2}\textup{ord}(J_{H,f})$. In case $(ii)$, we have $d_2=\frac{1}{2}(-\textup{ord}(J_{G})-4b)$.
\end{example} 

\begin{remark}
If we write $J_{G,f}$ and $J_{H,f}$ from Example \ref{ellipticcurvesintro} above in terms of $a$ and $b$, we obtain the result that $E$ has root configuration $(ii)$ if and only if $\frac{1}{6}\textup{ord}(\frac{1}{4a^3+27b^2})>\frac{1}{2}\textup{ord}\left(\frac{3^2a^2}{4a^3+27b^2}\right)$, which is equivalent to $\textup{ord}(j_E)<0$ when $K$ has residue characteristic $\neq 2, 3$. Since root configuration $(ii)$ is equivalent to $E/K$ having potentially multiplicative reduction (see, for example, \cite{omrisarah} Example 1.13), we recover the result that $E/K$ has potentially multiplicative reduction if and only if $\textup{ord}(j_E)<0$. 
\end{remark}

In the following theorem, we give a procedure that recovers the cluster picture from the list of rational functions in the roots of $f(x)$ associated to the weighted graphs in $\mathbf{G}_d$. For the definition of $G_n(f)$, $G_n'(f)$, $d_n(f)$ and $e_n(f)$ see Definitions \ref{introdefs1} and \ref{introdefs2} below. We explicitly write out the full algorithm for degree $5$ polynomials in \S\ref{degree5section}, where we give a table listing the graphs in $\mathbf{G}_5$, their associated rational functions and a description of how the configuration of the distances between the roots can be recovered from these. 

\begin{theorem}[=Theorem \ref{firstalgorithm}]\label{introthm}
Let $f(x)$ be a separable polynomial of degree $d$ over a discretely valued field $K$. 
\begin{enumerate}[(i)]
\item Given $G_{n}(f)$ and $d_1(f),\dots,d_{n}(f)$, let 
\begin{align}
    \mathcal{G}_{n+1}(f)=\{H\in\mathbf{G}_d:G_n'(f)=(V,E_n,w_n')\subsetneq H, &\text{ if } v_iv_j\not\in E_n \text{ then } w(v_iv_j)\in\{0,1\} \text{ and }\\
    & \text{ if } w(v_iv_j)\geq 1 \text{ and } w(v_jv_k)\geq 1 \text{ then } w(v_iv_k)\geq 1\},
\end{align}
and for $G\in\mathcal{G}_{n+1}(f)$ let 
\begin{equation}
            A_{n}(G,f)=\frac{-\textup{ord}(J_{G,f})-2\sum_{l=1}^{n}e_l(f)(n+2-l)d_l(f)}{2(|E_G|-|E_{G_{n}(f)}|)}.
        \end{equation}
Out of the graphs in $\mathcal{G}_{n+1}(f)$, let $G$ be the graph with the most edges satisfying $A_{n+1}(G,f)=\underset{H\in\mathcal{G}_{n+1}(f)}{\emph{\text{max}}} A_{n+1}(H,f)$. Then 
\begin{equation}
    G_{n+1}(f)=G \quad \textup{and} \quad d_{n+1}(f)=A_{n+1}(G,f). 
\end{equation}

\item Given $G_{n+1}(f)$ and $d_1(f),\dots, d_{n+1}(f)$, where $G_{n+1}(f)$ is the complete graph on $d$ vertices, fix a labelling of the vertices $v_1,\dots,v_d$ of $G_{n+1}(f)$. Then there is a labelling of the roots $x_1,\dots, x_d$ of $f(x)$ such that $\textup{ord}(x_i-x_j)=d_k(f)$ if and only if $w(v_iv_j)=n+2-k$ in $G_{n+1}(f)$. In particular, the set of tuples
\begin{equation}
\{(\textup{ord}(x_i-x_j),i,j) : 1\leq i<j \leq d \},
\end{equation}
up to reordering of the roots $x_1,\dots, x_d$, is uniquely determined by $G_{n+1}(f)$ and $d_1(f),\dots, d_{n+1}(f)$. 
\end{enumerate}
    
\end{theorem}

\begin{definition}[See Definition \ref{dnf}]\label{introdefs1}
Let $f(x)$ be a separable polynomial over a discretely valued field $K$ and let $x_1,\dots, x_d$ be the roots of $f(x)$ in $\Bar{K}$. Define $d_n(f)$ to be the $n$-th largest valuation in the set $\{\textup{ord}(x_i-x_j): i<j\}$, and define $k_f$ to be the number of distinct valuations in this set. Define $e_n(f)=\#\{\{x_i,x_j\}: i<j \textup{ and } \textup{ord}(x_i-x_j)=d_n(f)\}$.
\end{definition}

\begin{definition}[See Definitions \ref{aux} and \ref{curlyg}]\label{introdefs2}
Let $f(x)$ be a separable polynomial over a discretely valued field $K$ and fix a labelling $x_1,\dots, x_d$ of the roots of $f(x)$ in $\Bar{K}$. For $1\leq n\leq k_f$, define $G_n(f)=(V,E_n,w_n)$, where 
\begin{enumerate}[(i)]
\item $V=\{v_1,\dots,v_d\}$ and $E_n=\{v_iv_j:\textup{ord}(x_i-x_j)\geq d_n(f) \}$;
\item $w_n(v_iv_j)=n+1-m$ if $d_{m}(f)=\textup{ord}(x_i-x_j)$ for $m=1,\dots, n$.
\end{enumerate}
We consider $G_{n}(f)$ as an unlabelled graph and call it the \textit{$n$-th auxiliary graph} associated to $f(x)$. Define $G_{f,0}$ to be the empty graph on $d$ vertices and $G_n'(f)=(V,E_n,w_n')$, where $w_n'(v_iv_j)=n+2-m$ if $d_{m}(f)=\textup{ord}(x_i-x_j)$ for $m=1,\dots, n$. That is, $G_n'(f)$ is the $n$-th auxiliary graph associated to a cluster picture but with $1$ added to the weight of each of the edges.
\end{definition}

See Example \ref{2,3} for an example displaying the auxiliary graphs associated to a polynomial.

\subsection{Applications to hyperelliptic curves}\label{hyperellipticapplications}
Cluster pictures are defined for polynomials over discretely valued fields, but they are now a classical approach to studying the arithmetic of hyperelliptic curves over discretely valued fields. For a hyperelliptic curve $C:y^2=f(x)$ over a local field $K$, the cluster picture of $f(x)$ over $K$ can be used to calculate the curve's semistable model, conductor, minimal discriminant, Galois representation, Tamagawa number, root number, differential and more, and there is an exposition on how this can be done in \cite{usersguide}. In particular, there is a description of how the dual graph of the special fibre of the minimal strict normal crossings model, inertia action on the Tate module and conductor exponent can be obtained from the cluster picture and the valuation of the leading coefficient of $f(x)$ (see \cite{omrisarah} Theorem 1.2 and \cite{m2d2} Theorems 10.1 and 11.3). Combining these results with Theorem \ref{introthmfromg} above, we obtain the following theorem, which generalises the fact that when $K$ has residue characteristic $\geq 5$ one can obtain the Kodaira type of an elliptic curve $E/K^{\textup{unr}}$ from the valuation of $j_E$ and $\Delta_E$ (see, for example, \cite{advancedsilverman}). 

\begin{theorem}\label{hyperellipticinto}
Let $C:y^2=f(x)=c_f(x^d+a_{d-1}x^{d-1}+\cdots+a_0)$ be a hyperelliptic curve over a discretely valued field $K$. If $K$ is complete and has odd residue characteristic and $C/K$ has tame reduction, the valuations $\textup{ord}(J_{G,f})$ for all $G\in\mathbf{G}_d$ and $\textup{ord}(c_f)$ uniquely determine: 
\begin{enumerate}[(i)]
    \item The dual graph, with genus and multiplicity, of the special fibre of the minimal strict normal crossings model of $C/K^{\textup{unr}}$ if $C$ has genus $\geq 2$.
    \item The action of inertia on $V_\ell(\textup{Jac }C)=T_\ell(\textup{Jac }C)\otimes \mathbb{Q}_\ell$, where $T_\ell(\textup{Jac }C)$ is the $\ell$-adic Tate module of $\textup{Jac }C$ if $K$ is a local field. 
    \item The conductor exponent of $\textup{Jac } C$ if $K$ is a local field.
\end{enumerate}
\end{theorem}

\begin{remark}
It is important to highlight that Theorem \ref{hyperellipticinto} does not apply in the case where a wild extension is required for semistability.
\end{remark}

There have been numerous previous works on recovering the reduction type of curves from quantities that can be written in terms of the Weierstrass coefficients or the roots of $f(x)$. It was shown by Liu in \cite{liu} that for genus $2$ curves over a local fields the dual graph of the special fibre of their potential stable model can be recovered from the Igusa--Clebsch invariants of the curve \cite{igusa}. For genus $3$ hyperelliptic curves, there is a list of invariants describing their isomorphism classes given by Shioda in \cite{shioda} and Tsuyumine in \cite{tsuyumine}. In \cite{elisa} it is shown that the Shioda invariants can be expressed in terms of differences of roots of a Weierstrass equation and that this has applications to studying the reduction type of genus $3$ hyperelliptic curves. 

We highlight that the polynomials described in this paper are not invariants of the curve. Indeed, the cluster picture cannot be recovered from invariants of the curve since it is model dependent. There is a paper \cite{lilybelle} in preparation in which the author describes a list of invariants from which the dual graph of the special fibre of the mininmal regular model of a semistable hyperelliptic curve over a local field can be recovered:

\begin{theorem}[See \cite{lilybelle}]
There exists a finite and explicit list of invariants $I_{g}^{(1)},\dots,I_{g}^{(n_g)}$ for which the valuations $\textup{ord}(I_{g}^{(i)}(C))$ for $i=1,\dots,n_d$, when evaluated on a hyperelliptic curve $C$ of genus $g\geq 2$ over a local field of odd residue characteristic $K$, uniquely determine:
\begin{enumerate}[(i)]
    \item The dual graph of the special fibre of the minimal regular model of $C/K^{\textup{unr}}$ if $C$ is semistable;
    \item The dual graph of the special fibre of the potential stable model of $C/K^{\textup{unr}}$.
\end{enumerate}
\end{theorem}

\subsection{Layout of the paper}
This paper is laid out as follows. In \S\ref{notation} we list the notation used throughout.

In \S\ref{clusterpicssection} we give the definitions related to cluster pictures that are used in this paper and we state two important lemmas relating to auxiliary graphs associated to cluster pictures. We define the `averaging' function $A_n(G,f)$ associated to a weighted graph $G$ and a separable polynomial $f(x)$ over a discretely valued field $K$ that will be used to prove Theorem \ref{introthm}.

In \S\ref{valuationssection} we prove results that use the averaging function $A_{n}(-,f)$ to compare the valuations of the rational functions associated to different graphs when evaluated on the coefficients of a polynomial $f(x)$; this is the main ingredient of the proof of Theorem \ref{introthm} and will allow us to recover the cluster picture inductively.

In \S\ref{mainthmproof} we prove Theorem \ref{introthm} which describes how the cluster picture of $f(x)$ can be read off from the valuations of the rational functions from Definition \ref{invariants} when they are written in terms of and evaluated on the coefficients of $f(x)$. 

In \S\ref{degree5section}, we write out the algorithm that follows from Theorem \ref{introthm} and recovers the cluster picture of a degree $5$ polynomial over a discretely valued field. We explicitly write down the summands of the rational functions for the graphs in $\mathbf{G}_5$ and the depths of the cluster pictures in terms of the valuations of the rational functions. We give an example using this algorithm to calculate the cluster picture of a specific degree $5$ polynomial over $\mathbb{Q}_{7}$.

\subsection{Notation}\label{notation}
We will use the following notation. 
\begin{align}
& K && \text{ a discretely valued field}; \\
& \textup{ord} && \text{ the valuation with respect to a uniformiser of } K; \\
& \Bar{K} && \text{ the algebraic closure of } K; \\
& f(x) && \text{ a separable polynomial with coefficients in }K; \\
& \Delta && \text{ the discriminant of $f(x)/c_f$, where $c_f$ is the leading coefficient of $f(x)$}; \\
& G=(V,E,w) && \text{ a weighted graph with vertex set $V$ and edge set $E$, where } w:E\rightarrow \mathbb{Q}; \\
& vw && \text{ an edge between vertices } v \text{ and } w \textup{ in a graph $G$}; \\
& d_n(f) && \text{ the depth of the $n$-th deepest clusters, see Definition \ref{dnf};} \\
& k_f && \text{ the number of distinct depths in the cluster picture of $f(x)$, see Definition \ref{dnf}}; \\ 
& e_n(f) && \text{ the number of pairs of roots satisfying $\textup{ord}(x_i-x_j)=d_n(f)$, see Definition \ref{dnf};} \\
& G_n(f) && \text{ the $n$-th auxiliary graph of $f(x)$, see Definition \ref{aux}};\\
& \mathbf{G}_d && \text{ the set of auxiliary graphs on $d$ vertices, see Definition \ref{boldg}}; \\
& J_G && \text{ the rational function associated to a weighted graph, see Definition \ref{invariants}}; \\
& S_G^\sigma && \text{ a summand of $J_G$, see Definition \ref{invariants}}; \\
& J_{G,f} && \text{ $J_G$ evaluated on the roots of $f(x)$, see Definition \ref{invariants}}; \\
& S_{G,f}^\sigma && \text{ a summand of $J_G$ evaluated on the roots of $f(x)$, see Definition \ref{invariants}}; \\
& e && \text{ the identity element in } S_d/H\text{ where $H\leq S_d$}.
\end{align}

\noindent We adopt the convention that $v(0)=\infty$. All graphs will be considered as unlabelled unless otherwise stated, i.e. $G=(\{v_1,\dots,v_d\},\{v_1v_2\})$ is the same graph as $G'=(\{v_1,\dots,v_d\},\{v_1v_3\})$. 

\addtocontents{toc}{\SkipTocEntry}
	\section*{Acknowledgements} 

The author would like to thank Elisa Lorenzo García for posing this problem at the Seminari de Teoria de Nombres de Barcelona in 2022. She would also like to thank Vladimir Dokchitser for many useful discussions and the generosity of his support and guidance, and Holly Green for helpful comments.

The author was supported by the Engineering and Physical Sciences Research Council [EP/S021590/1], the EPSRC Centre for Doctoral Training in Geometry and Number Theory (The London School of Geometry and Number Theory), University College London.

\section{Cluster pictures and auxiliary graphs}\label{clusterpicssection}
We use the terminology of cluster pictures for the main definitions of this paper, so we recall the relevant definitions from \cite{m2d2} here. We also prove results on auxiliary graphs associated to cluster pictures, and define the `averaging function' $A_n(G,f)$ that will allow us to use the rational functions from Definition \ref{invariants} to recover the cluster picture. 

Throughout this section, we fix a separable polynomial $f(x)$ over a discretely valued field $K$. Cluster pictures are a pictorial object encoding the distances between the roots of $f(x)$. We write $\mathcal{R}$ for the set of roots of $f(x)$ in $\Bar{K}$ and $c_f$ for the leading coefficient of $f$ so that
\begin{equation}
    f(x)=c_f\prod_{r\in \mathcal{R}} (x-r).
\end{equation}

\begin{definition}[From \cite{m2d2} Definition 1.1]
A \textit{cluster} is a non-empty subset $\mathfrak{s}\subseteq\mathcal{R}$ of the form $\mathfrak{s}=D\cap\mathcal{R}$ for some disc $D = \{x \in \Bar{K} \mid \textup{ord}(x- z) \geq d\}$ for some $z \in \Bar{K}$ and $d \in \mathbb{Q}$.
\end{definition}

\begin{definition}[From \cite{m2d2} Definitions 1.1 and 1.4]
For a cluster $\mathfrak{s}$ with $|\mathfrak{s}| > 1$, its depth $d_{\mathfrak{s}}$ is the maximal $d$ for which $\mathfrak{s}$ is cut out by such a disc $D$ as above. That is,
\begin{equation}
    d_\mathfrak{s} = \text{{min}}_{r,r'\in\mathfrak{s}} \textup{ord}(r-r').
\end{equation}
If $\mathfrak{s}\neq \mathcal{R}$, then its relative
depth is $\delta_\mathfrak{s}=d_\mathfrak{s}-d_{P(\mathfrak{s})}$, where $P(\mathfrak{s})$ is the smallest cluster with $\mathfrak{s}\subsetneq P(\mathfrak{s})$. We refer to this data of the clusters and relative depths as the \textit{cluster picture} of $C$.
\end{definition}

\begin{remark}
Knowing the cluster picture is equivalent to knowing the set of tuples $\{(\textup{ord}(x_i-x_j),i,j) : 1\leq i< j \leq d \}$ up to reordering of the roots $x_1,\dots, x_d$ of $f(x)$.
\end{remark}

We draw cluster pictures by drawing the roots $r\in\mathcal{R}$ as red dots {\smash{\raise4pt\hbox{\clusterpicture 
\Root[D] {}{first}{r1} \endclusterpicture}}}, drawing ovals around the dots to represent clusters of size $>1$ and labelling the clusters with their relative depth $\delta_{\mathfrak{s}}$.

\begin{definition}\label{dnf}
Define $d_n(f)$ to be the depth of the $n$-th deepest clusters and $k_f$ to be the number of distinct depths in the cluster picture of $f(x)$. Define $e_n(f)=\#\{\{x_i,x_j\}: \ x_i,x_j\in \mathcal{R}, \ \textup{ord}(x_i-x_j)=d_n(f)\}$. 
\end{definition}

\begin{remark}
Note that definitions of $d_n(f)$, $k_f$ and $e_n(f)$ in Definition \ref{dnf} are consistent with those given in Definition \ref{introdefs1} which do not use the vocabulary of cluster pictures.
\end{remark}

\begin{example}\label{clusterpic1}
Let $f(x)=(x^2-p)((x-1)^2-p^2)((x+1)^2-p^2)$ over $\mathbb{Q}_p$. Then $\mathcal{R}=\{\pm\sqrt{p}, 1 \pm p, -1\pm p\}$ and $f(x)$ has the following cluster picture.
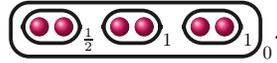
\begin{center}
\scalebox{1.75}{
    \clusterpicture             % elliptic curve IV*
  \Root[D] {1} {first} {r1};
  \Root[D] {} {r1} {r2};
    \Root[D] {3} {r2} {r3};
    \Root[D] {} {r3} {r4};
    \Root[D] {3} {r4} {r5};
    \Root[D] {} {r5} {r6};
  \ClusterD c1[\frac{1}{2}] = (r1)(r2);
  \ClusterD c2[1] = (r3)(r4);
  \ClusterD c3[1] = (r5)(r6);
  \ClusterD c4[0] = (c1)(c2)(c3);
\endclusterpicture}. 
\end{center}
Here, $k_f=3$ since there are three distinct depths $d_1(f)=1$, $d_2(f)=\frac{1}{2}$ and $d_3(f)=0$, and $e_1(f)=2$, $e_2(f)=1$ and $e_3(f)=12$.
\end{example}

We prove Theorem \ref{introthm} by constructing auxiliary graphs, to which we associate a rational function that is the quotient of symmetric polynomials in the roots of $f(x)$. From the valuations of these polynomials, we will recover the cluster picture by building it up inductively, from the deepest to shallowest clusters. This is done by inductively recovering $G_n(f)$, the definition for which we recall below.

\begin{definition}[=Definition \ref{introdefs2}]\label{aux}
Let $f(x)$ be a separable polynomial over a discretely valued field $K$ and fix a labelling $x_1,\dots, x_d$ of the roots of $f(x)$ in $\Bar{K}$. For $1\leq n\leq k_f$, define $G_n(f)=(V,E_n,w_n)$, where 
\begin{enumerate}[(i)]
\item $V=\{v_1,\dots,v_d\}$ and $E_n=\{v_iv_j:\textup{ord}(x_i-x_j)\geq d_n(f) \}$;
\item $w_n(v_iv_j)=n+1-m$ if $d_{m}(f)=\textup{ord}(x_i-x_j)$ for $m=1,\dots, n$.
\end{enumerate}
We consider $G_{n}(f)$ as an unlabelled graph and call it the \textit{$n$-th auxiliary graph} associated to $f(x)$. Define $G_{f,0}$ to be the empty graph on $d$ vertices.
\end{definition}

An edge of weight $1$ in $G_n(f)$ corresponds to two roots $x_i,x_j$ of $f(x)$ for which $\textup{ord}(x_i-x_j)=d_n(f)$, the depth of the $n$-th deepest clusters in the cluster picture. Those with weight $2$ correspond to two roots $x_i,x_j$ for which $\textup{ord}(x_i-x_j)=d_{n-1}(f)$, and so on.

\begin{example}\label{2,3}
Let us construct $G_n(f)$ for $n=0,\dots,k_f$ for a polynomial $f(x)$ over a discretely valued field $K$ with the following cluster picture
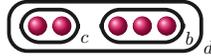
\begin{center}
\scalebox{1.75}{
\clusterpicture             % 5 roots, twin and triple 
  \Root[D] {1} {first} {r1};
  \Root[D] {} {r1} {r2};
  \Root[D] {3} {r2} {r3};
  \Root[D] {} {r3} {r4};
    \Root[D] {} {r4} {r5};
  \ClusterLD c1[][c] = (r1)(r2);
    \ClusterLD c2[][b] = (r3)(r4)(r5);
  \ClusterD c4[d] = (c2)(c1);
\endclusterpicture}
		\end{center} 
with $c>b$. Since {\smash{\raise2pt\hbox{\clusterpicture             % 2 roots
  \Root[D] {1} {first} {r1};
  \Root[D] {} {r1} {r2};
  \ClusterD c4[c] = (r1)(r2);
\endclusterpicture}}} is the deepest cluster, {\smash{\raise2pt\hbox{\clusterpicture             % 3 roots
  \Root[D] {1} {first} {r1};
  \Root[D] {} {r1} {r2};
  \Root[D] {} {r2} {r3};
  \ClusterD c4[b] = (r1)(r2)(r3);
\endclusterpicture}}} is the second deepest and {\smash{\raise2pt\hbox{\clusterpicture             % 3 roots
  \Root[D] {1} {first} {r1};
  \Root[D] {} {r1} {r2};
  \Root[D] {} {r2} {r3};
  \Root[D] {} {r3} {r4};
  \Root[D] {} {r4} {r5};
  \ClusterD c4[d] = (r1)(r2)(r3)(r4)(r5);
\endclusterpicture}}} is the third deepest, we have $k_f=3$, $d_{1}(f)=c+d$, $d_{2}(f)=b+d$ and $d_3(f)=d$. The zeroth, first, second and third auxiliary graphs are given below, where the unlabelled edges in $G_3(f)$ have weight $1$. 
\end{example}
\vspace{-15pt}
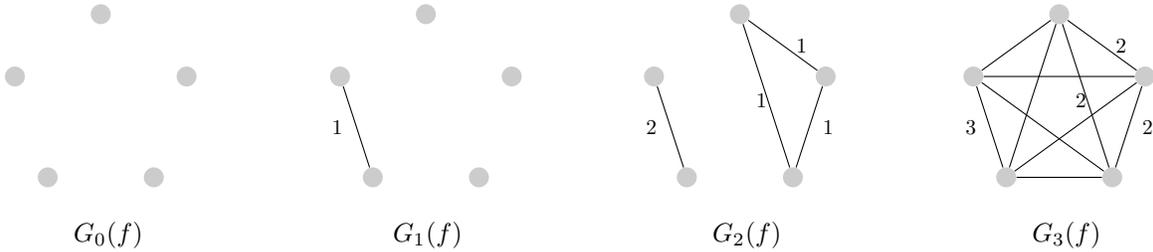
\begin{center}
\begin{minipage}[t]{0.2\textwidth}
		\begin{center}
				\begin{figure}[H]
			\begin{tikzpicture}
				[scale=0.4,auto=left,every node/.style={circle,fill=black!20,scale=0.8}]
				\node (n5) at (378:3)  {};
                \node (n2) at (306:3)  {};
                \node (n4) at (234:3)  {};
                \node (n3) at (162:3)  {};
				\node (n1) at (90:3) {};
			\end{tikzpicture}
 \caption*{$G_{0}(f)$}
		\end{figure}
		\end{center}
		\end{minipage}\hspace{25pt}
\begin{minipage}[t]{0.2\textwidth}
		\begin{center}
				\begin{figure}[H]
			\begin{tikzpicture}
				[scale=0.4,auto=left,every node/.style={circle,fill=black!20,scale=0.8}]
				\node (n5) at (378:3)  {};
                \node (n2) at (306:3)  {};
                \node (n4) at (234:3)  {};
                \node (n3) at (162:3)  {};
				\node (n1) at (90:3) {};
				
				\draw (n3) -- (n4) node [text=black,pos=0.5, left,fill=none] {$1$} ;

			\end{tikzpicture}
 \caption*{$G_{1}(f)$}
		\end{figure}
		\end{center}
		\end{minipage}\hspace{25pt}
		\begin{minipage}[t]{0.2\textwidth}
		\begin{center}
		\begin{figure}[H]
			\begin{tikzpicture}
				[scale=0.4,auto=left,every node/.style={circle,fill=black!20,scale=0.8}]
				\node (n5) at (378:3)  {};
                \node (n2) at (306:3)  {};
                \node (n4) at (234:3)  {};
                \node (n3) at (162:3)  {};
				\node (n1) at (90:3) {};
				
				\draw (n1) -- (n2) node [text=black,pos=0.4, below,fill=none] {$1$};
				\draw (n1) -- (n5) node [text=black,pos=0.5, right,fill=none] {$1$};
				\draw (n2) -- (n5) node [text=black,pos=0.5, right,fill=none] {$1$};
    
				\draw (n3) -- (n4) node [text=black,pos=0.5, left,fill=none] {$2$} ;
				
			\end{tikzpicture}
 \caption*{$G_{2}(f)$}
		\end{figure}
				\end{center}
				\end{minipage}\hspace{25pt}
		\begin{minipage}[t]{0.2\textwidth}
		\begin{center}
		\begin{figure}[H]
			\begin{tikzpicture}
				[scale=0.4,auto=left,every node/.style={circle,fill=black!20,scale=0.8}]
				\node (n5) at (378:3)  {};
                \node (n2) at (306:3)  {};
                \node (n4) at (234:3)  {};
                \node (n3) at (162:3)  {};
				\node (n1) at (90:3) {};
				
				\draw (n1) -- (n2) node [text=black,pos=0.4, below,fill=none] {$2$};
				\draw (n2) -- (n5) node [text=black,pos=0.5, right,fill=none] {$2$};
				
				\draw (n1) -- (n3);
				\draw (n1) -- (n4);
				\draw (n1) -- (n5) node [text=black,pos=0.5, right,fill=none] {$2$};
				\draw (n2) -- (n3);
				\draw (n2) -- (n4);
				\draw (n3) -- (n5);
				\draw (n4) -- (n5);
				
				\draw (n3) -- (n4) node [text=black,pos=0.5, left,fill=none] {$3$};
			\end{tikzpicture}
 \caption*{$G_{3}(f)$}
		\end{figure}
				\end{center}
				\end{minipage}
				\end{center}

In the Example \ref{2,3} above, if conversely we knew $G_{3}(f)$ and $d_{1}(f)$, $d_{2}(f)$ and $d_{3}(f)$, this would uniquely determine the cluster picture of $f(x)$ as being the one shown above, which we prove in Lemma \ref{clusterpicgraphs} below. The idea behind Theorem \ref{introthm} is to recover $G_{k_f}$ and $d_i(f)$ for $i=1,\dots,k_f$ using the rational functions from Definition \ref{invariants} for every $G\in\mathbf{G}_d$, so that the cluster picture can be recovered using this lemma. 

\begin{lemma}\label{clusterpicgraphs}
Let $f(x)$ be a separable polynomial over a discretely valued field $K$. Given $G_{k_f}(f)$ and $d_1(f),\dots, d_{k_f}(f)$, where $G_{k_f}(f)$ is the complete graph on $d$ vertices, fix a labelling of the vertices $v_1,\dots,v_d$ of $G_{k_f}(f)$. Then there is a labelling of the roots $x_1,\dots, x_d$ of $f(x)$ such that $\textup{ord}(x_i-x_j)=d_k(f)$ if and only if $w(v_iv_j)=k_f+1-k$ in $G_{k_f}(f)$. In other words, the set of tuples
\begin{equation}
\{(\textup{ord}(x_i-x_j),i,j) : 1\leq i<j \leq d \},
\end{equation}
up to reordering of the roots $x_1,\dots, x_d$ of $f(x)$, and thus the cluster picture, is uniquely determined by $G_{k_f}(f)$ and $d_1(f),\dots, d_{k_f}(f)$.
\end{lemma}

\begin{proof}
By the definition of the auxiliary graph $G_{k_f}(f)$ (Definition \ref{aux}), the edges of $G_{k_f}(f)$ of weight $k_f+1-k$ correspond to tuples of roots $x_i,x_j$ for which $\textup{ord}(x_i-x_j)=d_k(f)$. Thus, if $d_k(f)$ is known for $1\leq m\leq k_f$, this information uniquely determines the cluster picture.
\end{proof}

The following lemma tells us that the auxiliary graphs are a disjoint union of complete graphs.

\begin{lemma}\label{complete}
Let $f(x)$ be a separable polynomial over a discretely valued field $K$. The $n$-th auxiliary graph $G_n(f)$ is a disjoint union of complete graphs.
\end{lemma}

\begin{proof}
Fix a labelling $x_1,\dots, x_d$ of the roots of $f(x)$ corresponding to a labelling $v_1,\dots,v_d$ of the vertices of $G_n(f)$. If $\textup{ord}(x_i-x_j)=a$ and $\textup{ord}(x_j-x_k)=b$ then $\textup{ord}(x_i-x_k)\geq \min(a,b)$, since $K$ is a discretely valued field. So if $v_iv_j\in G_n(f)$ and $v_jv_k\in G_n(f)$ then $v_iv_k\in G_n(f)$, whence $G_n(f)$ is a disjoint union of complete graphs.
\end{proof}

\begin{definition}\label{curlyg}
Let $f(x)$ be a separable polynomial over a discretely valued field, where $f(x)$ has degree $d$. Let $G_n'(f)=(V,E_n,w_n')$, where $w_n'(v_iv_j)=n+2-m$ if $d_{m}(f)=\textup{ord}(x_i-x_j)$ for $m=1,\dots, n$. That is, $G_n'(f)$ is the $n$-th auxiliary graph associated to a cluster picture but with $1$ added to the weight of each of the non-zero weighted edges. Define
\begin{align}
    \mathcal{G}_{n+1}(f)=\{H=(V,E,w):G_n'(f)\subsetneq H, &\text{ if } v_iv_j\not\in E_n \text{ then } w(v_iv_j)\in\{0,1\} \text{ and }\\
    & \text{ if } w(v_iv_j)\geq 1 \text{ and } w(v_jv_k)\geq 1 \text{ then } w(v_iv_k)\geq 1\}.
\end{align}

In other words, $\mathcal{G}_{n+1}(f)$ is the set of graphs on $d$ vertices that strictly contain $G_n'(f)$ as a subgraph, that are the union of complete graphs, and for which the edges not in $G_n'(f)$ have weight $1$. 
\end{definition}

Lemma \ref{complete} tells us that $G_{n+1}(f)$ is a disjoint union of complete graphs, so the set $\mathcal{G}_{n+1}(f)$ contains all possibilities for $G_{n+1}(f)$, given that $G_n(f)$ is already known.

\begin{example}\label{curlygex}
Suppose $f(x)$ over $K$ has degree $5$ and assume we know that $G_{1}(f)$ is the following graph.
		
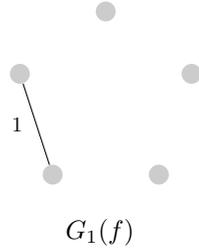
\begin{figure}[H]
			\begin{tikzpicture}
				[scale=0.4,auto=left,every node/.style={circle,fill=black!20,scale=0.8}]
				\node (n5) at (378:3)  {};
                \node (n2) at (306:3)  {};
                \node (n4) at (234:3)  {};
                \node (n3) at (162:3)  {};
				\node (n1) at (90:3) {};
				
				\draw (n3) -- (n4) node [text=black,pos=0.5, left,fill=none] {$1$} ;
			\end{tikzpicture}
 \caption*{$G_{1}(f)$}
		\end{figure}
\noindent Then $\mathcal{G}_{2}(f)$ contains precisely the following graphs, where the unlabelled edges have weight $1$. In other words, the possibilities for $G_{2}(f)$ are:

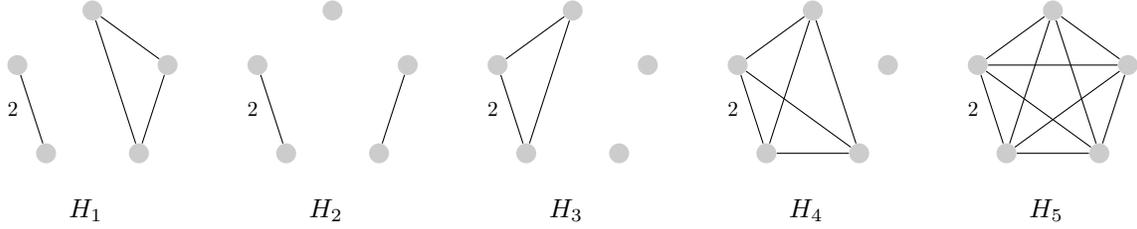
\begin{center}

\begin{minipage}[t]{0.2\textwidth}
		\begin{center}
		\begin{figure}[H]

			\begin{tikzpicture}
				[scale=0.35,auto=left,every node/.style={circle,fill=black!20,scale=0.8}]
				\node (n5) at (378:3)  {};
                \node (n2) at (306:3)  {};
                \node (n4) at (234:3)  {};
                \node (n3) at (162:3)  {};
				\node (n1) at (90:3) {};
				
				\draw (n1) -- (n2);
				\draw (n1) -- (n5) ;
				\draw (n2) -- (n5);
				\draw (n3) -- (n4) node [text=black,pos=0.5, left,fill=none] {$2$};
			\end{tikzpicture}
 \caption*{$H_1$}
		\end{figure}
				\end{center}
		 
\end{minipage}\hspace{-5pt}
\begin{minipage}[t]{0.2\textwidth}
		\begin{center}
		\begin{figure}[H]

			\begin{tikzpicture}
				[scale=0.35,auto=left,every node/.style={circle,fill=black!20,scale=0.8}]
				\node (n5) at (378:3)  {};
                \node (n2) at (306:3)  {};
                \node (n4) at (234:3)  {};
                \node (n3) at (162:3)  {};
				\node (n1) at (90:3) {};

				\draw (n2) -- (n5);
				\draw (n3) -- (n4) node [text=black,pos=0.5, left,fill=none] {$2$};
			\end{tikzpicture}
 \caption*{$H_2$}
		\end{figure}
				\end{center}
		 
\end{minipage}\hspace{-5pt}
\begin{minipage}[t]{0.2\textwidth}
		\begin{center}
		\begin{figure}[H]

			\begin{tikzpicture}
				[scale=0.35,auto=left,every node/.style={circle,fill=black!20,scale=0.8}]
				\node (n5) at (378:3)  {};
                \node (n2) at (306:3)  {};
                \node (n4) at (234:3)  {};
                \node (n3) at (162:3)  {};
				\node (n1) at (90:3) {};
				
				\draw (n1) -- (n3);
				\draw (n1) -- (n4) ;
				\draw (n3) -- (n4) node [text=black,pos=0.5, left,fill=none] {$2$};
			\end{tikzpicture}
 \caption*{$H_3$}
		\end{figure}
				\end{center}
		 
\end{minipage}\hspace{-5pt}
\begin{minipage}[t]{0.2\textwidth}
		\begin{center}
		\begin{figure}[H]

			\begin{tikzpicture}
				[scale=0.35,auto=left,every node/.style={circle,fill=black!20,scale=0.8}]
				\node (n5) at (378:3)  {};
                \node (n2) at (306:3)  {};
                \node (n4) at (234:3)  {};
                \node (n3) at (162:3)  {};
				\node (n1) at (90:3) {};
				
				\draw (n1) -- (n3);
				\draw (n1) -- (n4);
				\draw (n1) -- (n2);
                \draw (n3) -- (n2);
                \draw (n4) -- (n2) ;
				\draw (n3) -- (n4) node [text=black,pos=0.5, left,fill=none] {$2$};
			\end{tikzpicture}
 \caption*{$H_4$}
		\end{figure}
				\end{center}
		 
\end{minipage}\hspace{-5pt}
\begin{minipage}[t]{0.2\textwidth}
		\begin{center}
		\begin{figure}[H]

			\begin{tikzpicture}
				[scale=0.35,auto=left,every node/.style={circle,fill=black!20,scale=0.8}]
				\node (n5) at (378:3)  {};
                \node (n2) at (306:3)  {};
                \node (n4) at (234:3)  {};
                \node (n3) at (162:3)  {};
				\node (n1) at (90:3) {};
				
				\draw (n1) -- (n3);
				\draw (n1) -- (n4);
				\draw (n1) -- (n2);
                \draw (n3) -- (n2);
                \draw (n4) -- (n2);
				\draw (n1) -- (n5);
				\draw (n2) -- (n5) ;
				\draw (n3) -- (n4)  node [text=black,pos=0.5, left,fill=none] {$2$};
				\draw (n3) -- (n5) ;
				\draw (n4) -- (n5) ;
			\end{tikzpicture}
 \caption*{$H_5$}
		\end{figure}
				\end{center}
\end{minipage}
\end{center}
\end{example}

The idea behind the proof of Theorem \ref{introthm} is that the valuation of the rational functions $J_{H_i,f}$ from Definition \ref{invariants} for $i=1,\dots, 5$, given that $G_1(f)$ and $d_1(f)$ are known, will uniquely determine $G_{2}(f)$ and $d_{2}(f)$. The key to determining which $H_{i}$ in the example above is $G_{2}(f)$ is the following definition which shifts and scales the valuation of $J_{G,f}$ for $G\in\mathcal{G}_{n+1}(f)$ based on the values of $G_{n}(f)$ and $d_1(f),\dots,d_n(f)$.

\begin{definition}\label{average}
For $H\in \mathcal{G}_{n+1}(f)$, define
\begin{equation}
        A_{n+1}(H,f)=\frac{-\textup{ord}(J_{H,f})-2\sum_{l=1}^{n}e_l(f)(n+2-l)d_l(f)}{2( |E_H|-|E_{G_n(f)}|)},
\end{equation}
where $|E_H|$ and $|E_{G_n(f)}|$ denote the number of edges of $H$ and $G_n(f)$ respectively and $e_i(f)$ is as in Definition \ref{dnf}. If $J_{H,f}=0$ then we adopt the convention that $\textup{ord}(J_{H,f})=\infty$ and $A_{n+1}(H,f)=-\infty$. 
\end{definition}

\begin{example}
Let us demonstrate how $A_2(H_i,f)$ for $H_i$ in Example \ref{curlygex} can be used to calculate $G_2(f)$ for a polynomial $f(x)$ over a discretely valued field $K$ with cluster picture 
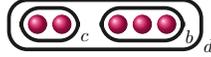
\begin{center}
\scalebox{1.75}{
\clusterpicture             % 5 roots, twin and triple 
  \Root[D] {1} {first} {r1};
  \Root[D] {} {r1} {r2};
  \Root[D] {3} {r2} {r3};
  \Root[D] {} {r3} {r4};
    \Root[D] {} {r4} {r5};
  \ClusterLD c1[][c] = (r1)(r2);
    \ClusterLD c2[][b] = (r3)(r4)(r5);
  \ClusterD c4[d] = (c2)(c1);
\endclusterpicture}
		\end{center} 
where $c>b$. Suppose we know $G_{1}(f)$ but we are trying to work out $G_{2}(f)$ from the valuation of the rational functions associated to the possibilities for $G_2(f)$, rather than reading it off the cluster picture. The graphs in $\mathcal{G}_{2}(f)$ are the graphs $H_1$ to $H_5$ in Example \ref{curlygex}. For each graph $H_i$, we can write down the the associated `average' $A_{2}(H_i,f)$ in terms of $b$, $c$ and $d$ by studying the valuation of $J_{H_i,f}$ and $J_{G_1'(f),f}$, and they are written in the table below. To study these valuations, we can use the fact that we know the cluster picture a priori, and therefore we know the valuation of the differences of the roots. 

\begin{center}
\begin{tabular}{|c||c|c|c|c|c|}
    \hline
      $H_i$ &  $H_1$ & $H_2$ & $H_3$ & $H_4$ & $H_5$ \\
        \hline 
       $A_{2}(H_i,f)$  & $b+d$ & $\leq b+d$ & $\max(d,2b+d-c)$ & $\leq\frac{b}{5}+d$ & $\leq \frac{b}{3}+d$ \\
       \hline 
\end{tabular}
\end{center}

Out of the graphs in $\mathcal{G}_{2}(f)$, $H_1$ is the unique graph among those with the highest value of $A_{2}(H_i,f)$ that has the most edges: $A_{2}(H_1,f)$ and $A_{2}(H_2,f)$ may have the same value, but $H_1$ has more edges than $H_2$. It turns out that this is telling us that $G_{2}(f)=H_1$ and $d_{2}(f)=A_{2}(H_1,f)=b+d$, as we would expect since we already know the cluster picture. This is the idea behind Theorem \ref{introthm}, which gives the general statement on finding $G_{n+1}(f)$ given that $G_n(f)$ is known by looking at the value of $A_{n+1}(H,f)$ for each $H\in\mathcal{G}_{n+1}(f)$.
\end{example}

\section{Comparing the valuations of the rational functions}\label{valuationssection}
In this section we prove general results that use $A_{n+1}(-,f)$ to compare the valuations of rational functions associated to different possibilities for $G_{n+1}(f)$, given that $G_n(f)$ and $d_1(f),\dots,d_n(f)$ are known. It is these comparisons that will allow us to recover $G_{n+1}(f)$ and $d_{n+1}(f)$ and prove Theorem \ref{introthm}. Throughout this section, we fix a separable polynomial $f(x)$ of degree $d$ over a discretely valued field $K$. We will first need the following lemma, which tells us that it makes sense to take the valuation of the rational functions $J_{G,f}$, for $G$ a weighted graph on $d$ vertices, since they are not identically zero.

\begin{lemma}\label{notzero}
Let $G=(V,E,w)$ be a weighted graph on $d$ vertices. Then the rational function $J_G$ in the variables $X_1,\dots, X_d$ is not identically zero.  
\end{lemma}

\begin{proof}
Once the summands $S_{G}^\sigma$ of $J_G$ are put over a common denominator, the monomials  in the numerator for which each variable has an even exponent appear with positive coefficients.
\end{proof}

\begin{remark}
Although the rational functions are not identically zero, it could happen that when evaluated on a polynomial $f(x)$ we have $J_{H,f}=0$ for some $H\in \mathbf{G}_{\deg(f)}$. We adopt the convention that $v(0)=\infty$. 
\end{remark}

We will use the following lemma throughout this section when proving results on the valuations of the rational functions. It tells us that we can think of the valuation of $S_{G,f}^\sigma$ as a sum that allocates the weights of the edges of $G$ to the numbers in the sequence of depths $d_{1}(f)>d_{2}(f)>\dots>d_{k_f}(f)$.

\begin{lemma}\label{sweights}
Let $G=(V,E,w)$ be a weighted graph on $d$ vertices. Fix a labelling of the vertices $v_1,\dots,v_d\in V$ corresponding to the roots $x_1,\dots,x_d$ of $f(x)$, and let $\sigma\in S_d/\textup{Stab}_{S_d}G$. Then 
\begin{equation}
    \textup{ord}(S_{G,f}^{\sigma})=-2\left(\sum_{v_iv_j\in \sigma(E)}\textup{ord}(x_i-x_j)w(v_{\sigma^{-1}(i)}v_{\sigma^{-1}(j)})\right). 
\end{equation}
\end{lemma}

\begin{proof}
This follows immediately from the definition of $S_{G,f}^{\sigma}$.
\end{proof}

In order to prove the results in this section, we will need the following well known fact. 

\begin{fact}\label{weights}
Let $x_1\geq x_2\geq \dots \geq x_k$ and $w_1\geq w_2\geq \dots \geq w_{k}$ be two descending sequences of rational numbers. Let $\sigma$ be a permutation of $1,\dots,k$ for which $(w_1,\dots,w_k)\neq (w_{\sigma(1)},\dots,w_{\sigma(k)})$. Then 
\begin{equation}
    \sum_{i=1}^{k}w_ix_i > \sum_{i=1}^{k} w_{\sigma(i)}x_i.
\end{equation}
That is to say, the sum is maximised by allocating the highest weight $w_1$ to the highest number $x_1$, the second highest to the second highest and so on. 
\end{fact}

We will utilise this fact since we need to compare the values of $\textup{ord}(J_{G,f})$ for all $G\in \mathcal{G}_{n+1}(f)$, so we need to know when $-\textup{ord}(S_{G,f}^\sigma)$ is maximised. We first prove the following lemma, which gives us the part of Theorem \ref{introthm} that tells us the value of the $(n+1)$-st greatest depth. 

\begin{lemma}\label{adn}
Let $G_{n+1}(f)$ be the $(n+1)$-st auxiliary graph for $n\geq 0$. Then $d_{n+1}(f)=A_{n+1}(G_{n+1}(f),f)$. 
\end{lemma}

\begin{proof}
Fix a labelling of the vertices of $G_{n+1}(f)$ so that $w(v_iv_j)=n+2-m$ if and only if $\textup{ord}(x_i-x_j)=d_{m}(f)$ for $m=1,\dots,n+1$. This is possible by the definition of $G_{n+1}(f)$. Then, also by the definition of $G_{n+1}(f)$, for $n\geq 0$
\begin{equation}
    \textup{ord}(S_{G_{n+1}(f),f}^e)=-2\sum_{l=1}^{n+1} e_l(f)(n+2-l)d_{l}(f),
\end{equation}
where, as in Definition \ref{dnf}, $e_l(f)=\#\{\{x_i,x_j\}: \ x_i,x_j\in \mathcal{R}, \ \textup{ord}(x_i-x_j)=d_l(f)\}$. By Fact \ref{weights}, this is the unique summand of $J_{G_{n+1}(f),f}$ with the lowest valuation, since all other summands allocate a lower weight to greater depths. Hence $\textup{ord}(J_{G_{n+1}(f),f})=\textup{ord}(S_{G_{n+1}(f),f}^e)$. Since $e_{n+1}(f)=|E_{G_{n+1}(f)}|-|E_{G_n(f)}|$ by Definition \ref{aux}, this gives us 
\begin{equation}
    \textup{ord}(J_{G_{n+1}(f),f})=-2\sum_{l=1}^{n}e_l(f)(n+2-l)d_l(f)-d_{n+1}(f)(2\cdot |E_{G_{n+1}(f)}|- 2\cdot |E_{G_n(f)}|)
\end{equation}
for $n\geq 0$, and so
\begin{equation}
    A_{n+1}(G_{n+1}(f),f)=\frac{-\textup{ord}(J_{G_{n+1}(f),f})-2\sum_{l=1}^{n}e_i(f)(n+2-l)d_i(f)}{(2\cdot |E_{G_{n+1}(f)}|- 2\cdot |E_{G_n(f)}|)} = d_{n+1}(f).
\end{equation}
\end{proof}

We can now proceed to prove the series of results that compares the values of $A_{n+1}(G,f)$ for $G\in\mathcal{G}_{n+1}(f)$. 

\begin{lemma}\label{samenoedges}
Let $G\in\mathcal{G}_{n+1}(f)$ with $G\neq G_{n+1}(f)$. Suppose $G$ and $G_{n+1}(f)$ have the same number of edges. Then 
\begin{equation}
    \textup{ord}(J_{G,f})>\textup{ord}(J_{G_{n+1}(f),f}).
\end{equation}
\end{lemma}

\begin{proof}
We know from the proof of Lemma \ref{adn} that
\begin{equation}
    \textup{ord}(J_{G_{n+1}(f),f})=-2\sum_{l=1}^{n+1} e_l(f)(n+2-l)d_{l}(f),
\end{equation}
where $e_l(f)=\#\{\{x_i,x_j\}: \ x_i,x_j\in \mathcal{R},\ i<j \textup{ and } \textup{ord}(x_i-x_j)=d_l(f)\}$. It comes from allocating weight $n+1$ to the edges $v_iv_j$ with $\textup{ord}(x_i-x_i)=d_{1}(f)$, weight $n$ to the edges corresponding to vertices with depth $d_{2}(f)$, and so on. Fix a labelling of the vertices of $G$ corresponding to the roots $x_1,\dots,x_d$ of $f(x)$ so that $S_{G,f}^{e}$ is a summand of $J_{G,f}$ with the lowest valuation (there may be multiple summands with the same valuation). If $\textup{ord}(S_{G,f}^{e})$ does not allocate weight $n+2-l$ to depths $d_{l}(f)$ for $l=1,\dots n$, by Fact \ref{weights}, $\textup{ord}(J_{G,f})\geq \textup{ord}(S_{G,f}^{e})>\textup{ord}(J_{G_{n+1}(f)})$. If $\textup{ord}(S_{G,f}^{e})$ allocates weight $n+2-l$ to depths $d_{l}(f)$ for $l=1,\dots, n$, then 
\begin{equation}
    S_{G,f}^{e}(f)= -2\sum_{i=l}^{n} e_l(f)(n+2-l)d_{l}(f) - 2r_{n+1} d_{n+1}(f)-2r_{n+2} d_{n+2}(f)-\dots-2r_{n+t} d_{n+t}(f),
\end{equation}
where $r_i$ is the number of pairs of roots $\{x_i,x_j\}$ for which $w(v_iv_j)=1$ and $\textup{ord}(x_i-x_j)=d_{i}(f)$ for $i=n+1,\dots ,n+t$ and $r_{n+1}+\dots+r_{n+t}=e_{n+1}(f)$ is the number of edges of weight $1$ in $G$. It is clear that $r_{n+1}\leq e_{n+1}(f)$; we want to show that $r_{n+1}<e_{n+1}(f)$. 

Let, in some labelling of the vertices and roots, $v_1v_2$, $v_3v_4$,\dots, $v_{m}v_{m+1}\in G_{n+1}(f)$ be the $e_{n+1}(f)$ edges with weight $1$, possibly with some $v_i=v_{i+1}$. So $\textup{ord}(x_1-x_2)=\dots=\textup{ord}(x_m-x_{m+1})=d_{n+1}(f)$. For a contradiction, suppose $r_{n+1}=e_{n+1}(f)$. This would mean that $v_1v_2$, $v_3v_4$,\dots, $v_{m}v_{m+1}\in G$ so that in the chosen labelling of the vertices $G=G_{n+1}(f)$. This is a contradiction since we assumed $G\neq G_{n+1}(f)$ as unlabelled graphs. Thus, $r_{n+1}<e_{n+1}(f)$ and so $\textup{ord}(J_{G,f})\geq \textup{ord}(S_{G,f}^{e})>\textup{ord}(J_{G_{n+1}(f),f})$.
\end{proof}

\begin{lemma}\label{agraphs}
Suppose $G\in \mathcal{G}_{n+1}(f)$ and $G\neq G_{n+1}(f)$. 
\begin{enumerate}[(i)]
    \item If $G$ has fewer edges than $G_{n+1}(f)$ then $A_{n+1}(G_{n+1}(f),f)\geq A_{n+1}(G,f)$.
    \item If $G$ has the same number or more edges than $G_{n+1}(f)$ then $A_{n+1}(G_{n+1}(f),f)>A_{n+1}(G,f)$. 
\end{enumerate}
\end{lemma}

\begin{proof}
(i) Suppose $G$ has fewer edges than $G_{n+1}(f)$. Let $S_{G,f}^{\sigma}$ be a summand of $J_{G,f}$ with the lowest valuation (there may be multiple summands with the same valuation). Since $G$ contains $G_n(f)$, by Fact \ref{weights}, for $n\geq 1$
\begin{equation}
    \textup{ord}(J_{G,f})\geq \textup{ord}(S_{G,f}^{\sigma})\geq -2\sum_{l=1}^{n} e_l(f)(n+2-l)d_{l}(f) - 2(|E_G|-|E_{G_n(f)}|)d_{n+1}(f).
\end{equation}
Thus
\begin{equation}
    A_{n+1}(G,f)=\frac{-\textup{ord}(J_{G,f})-2\sum_{l=1}^{n}e_l(f)(n+2-l)d_l(f)}{(2\cdot |E_G|- 2\cdot |E_{G_n(f)}|)} \leq d_{n+1}(f) =A_{n+1}(G_{n+1}(f),f)
\end{equation}
by Lemma \ref{adn}.

(ii) Suppose that $G$ has the same number of edges as $G_{n+1}(f)$. We know from Lemma \ref{samenoedges} that $\textup{ord}(J_{G,f})>\textup{ord}(J_{G_{n+1}(f)})$ and so $A_{n+1}(G_{n+1}(f),f)>A_{n+1}(G,f)$ since $|E_G|=|E_{G_{n+1}(f)}|$. Finally, suppose $G$ has more edges than $G_{n+1}(f)$. Fix a labelling of the vertices of $G$ corresponding to the roots $x_1,\dots,x_d$ of $f(x)$ so that $S_{G,f}^{e}$ is a summand of $J_{G,f}$ with the lowest valuation. If $\textup{ord}(S_{G,f}^{e})$ does not allocate weight $n+2-l$ to depths $d_{l}(f)$ for $l=1,\dots n$, by Fact \ref{weights},
\begin{align}
    \textup{ord}(J_{G,f}) \geq \textup{ord}(S_{G,f}^{e})&>\textup{ord}(J_{G_{n+1}(f)})- d_{n+1}(f)(2\cdot |E_{G}|- 2\cdot |E_{G_{n+1}(f)}|) \\
    &=\textup{ord}(J_{G_{n}'(f)})- d_{n+1}(f)(2\cdot |E_{G}|- 2\cdot |E_{G_{n}(f)}|)
\end{align}
If $\textup{ord}(S_{G,f}^{e})$ allocates weight $n+2-l$ to depths $d_{l}(f)$ for $l=1,\dots, n$, then
\begin{equation}
    \textup{ord}(S_{G,f}^{e})= -2\sum_{l=1}^{n} e_l(f)(n+2-l)d_{l}(f) - 2r_{n+1}d_{n+1}(f)-2r_{n+2}d_{n+2}(f)-\dots - 2r_{n+t}d_{n+t}(f).
\end{equation}  
Here, $r_m$ is the number of pairs of roots $\{x_i,x_j\}$ for which $w(v_iv_j)=1$ and $\textup{ord}(x_i-x_j)=d_{m}(f)$ for $m=n+1,\dots n+t$, and $r_{n+1}+\dots+r_{n+t}=|E_{G}|-|E_{G_n(f)}|$ is the number of edges of weight $1$ in $G$. We cannot have $r_{n+1}=e_{n+1}(f)=|E_{G_{n+1}(f)}|-|E_{G_n(f)}|$ and $r_{n+1}=\cdots=r_{n+t}=0$, where $e_{n+1}(f)$ is the number of edges of weight $1$ in $G_{n+1}(f)$, since this would imply that $r_{n+1}+\dots+r_{n+t}=|E_G|-|E_{G_n(f)}|=|E_{G_{n+1}(f)}|-|E_{G_n(f)}|$, which contradicts the fact that we assumed $|E_G|>|E_{G_{n+1}(f)}|$. Hence 
\begin{equation}
    \textup{ord}(J_{G,f}) \geq \textup{ord}(S_{G,f}^{e}) > -2\sum_{l=1}^{n} e_l(f)(n+2-l)d_{l}(f) - d_{n+1}(f)(2\cdot |E_{G}|- 2\cdot |E_{G_n(f)}|)
\end{equation}
since $d_{n+1}(f)>\dots>d_{n+t}(f)$. Thus, in either case  
\begin{equation}
    A_{n+1}(G,f)=\frac{-\textup{ord}(J_{G,f})-2\sum_{l=1}^{n} e_l(f)(n+2-l)d_{l}(f)}{(2\cdot |E_G|- 2\cdot |E_{G_n(f)}|)} < d_{n+1}(f) = A_{n+1}(G_{n+1}(f),f),
\end{equation}
which finishes the proof.
\end{proof}

We can now prove the most important result of this section, which tells us how we can distinguish $G_{n+1}(f)$ from all the other possibilities for $G_{n+1}(f)$ using the `averaging' function $A_{n+1}(-,f)$. This theorem gives us part $(i)$ of Theorem \ref{introthm}, and allows us to prove Theorems \ref{top} and \ref{introthmfromg} in the next section.

\begin{theorem}\label{main}
Let $G\in\mathcal{G}_{n+1}(f)$ be such that $A_{n+1}(G,f)=\underset{H\in\mathcal{G}_{n+1}(f)}{\emph{\text{max}}} A_{n+1}(H,f)$, where $G$ has the most edges out of all such graphs. Then $G_{n+1}(f)=G$ and $d_{n+1}(f)=A_{n+1}(G,f)$.
\end{theorem}

\begin{proof}
Lemma \ref{complete} tells us that $G_{n+1}(f)$ is a disjoint union of complete graphs, so $G_{n+1}(f)\in \mathcal{G}_{n+1}(f)$. By Lemma \ref{agraphs}, $G_{n+1}(f)$ is the unique graph in $\mathcal{G}_{n+1}(f)$ that maximises $A_{n+1}(-,f)$ and has the most edges out of such graphs. By Lemma \ref{adn}, $d_{n+1}(f)=A_{n+1}(G_{n+1}(f),f)$.
\end{proof}

\section{Recovering the cluster picture from the rational functions}\label{mainthmproof}

Fix a separable polynomial $f(x)$ over a discretely valued field $K$ and let $d=\textup{deg}(f)$. In this section we restate Theorem \ref{introthm} from \S\ref{intro} that describes the procedure by which the cluster picture of $f(x)$ can be recovered from rational functions in the coefficients; it follows immediately from Theorem \ref{main} and Lemma \ref{clusterpicgraphs}. We explain how to write the rational functions in terms of the coefficients of $f(x)$, and explicitly describe the list of polynomials that uniquely determine the cluster picture from Theorem \ref{top}. 

\begin{theorem}\label{firstalgorithm}
Let $f(x)$ be a separable polynomial of degree $d$ over a discretely valued field $K$. 
\begin{enumerate}[(i)]
\item Given $G_{n}(f)$ and $d_1(f),\dots,d_{n}(f)$, let 
\begin{align}
    \mathcal{G}_{n+1}(f)=\{H\in\mathbf{G}_d:G_n'(f)=(V,E_n,w_n')\subsetneq H, &\text{ if } v_iv_j\not\in E_n \text{ then } w(v_iv_j)\in\{0,1\} \text{ and }\\
    & \text{ if } w(v_iv_j)\geq 1 \text{ and } w(v_jv_k)\geq 1 \text{ then } w(v_iv_k)\geq 1\},
\end{align}
and for $G\in\mathcal{G}_{n+1}(f)$ let 
\begin{equation}
            A_{n}(G,f)=\frac{-\textup{ord}(J_{G,f})-2\sum_{l=1}^{n}e_l(f)(n+2-l)d_l(f)}{2(|E_G|-|E_{G_{n}(f)}|)}.
        \end{equation}
Out of the graphs in $\mathcal{G}_{n+1}(f)$, let $G$ be the graph with the most edges satisfying $A_{n+1}(G,f)=\underset{H\in\mathcal{G}_{n+1}(f)}{\emph{\text{max}}} A_{n+1}(H,f)$. Then 
\begin{equation}
    G_{n+1}(f)=G \quad \textup{and} \quad d_{n+1}(f)=A_{n+1}(G,f). 
\end{equation}

\item Given $G_{n+1}(f)$ and $d_1(f),\dots, d_{n+1}(f)$, where $G_{n+1}(f)$ is the complete graph on $d$ vertices, fix a labelling of the vertices $v_1,\dots,v_d$ of $G_{n+1}(f)$. Then there is a labelling of the roots $x_1,\dots, x_d$ of $f(x)$ such that $\textup{ord}(x_i-x_j)=d_k(f)$ if and only if $w(v_iv_j)=n+2-k$ in $G_{n+1}(f)$. In particular, the set of tuples
\begin{equation}
\{(\textup{ord}(x_i-x_j),i,j) : 1\leq i< j \leq d \},
\end{equation}
up to reordering of the roots $x_1,\dots, x_d$, is uniquely determined by $G_{n+1}(f)$ and $d_1(f),\dots, d_{n+1}(f)$. 
\end{enumerate}
\end{theorem}

Theorem \ref{firstalgorithm} tells us that we can recover the whole cluster picture inductively, starting with $G_0(f)$ and recovering $G_1(f)$ and $d_1(f)$, up to finding $G_{k_f}(f)$ and $d_{k_f}(f)$.

\begin{definition}\label{listofpolys}
By the definition of $J_{G}$ (Definition \ref{invariants}), for every $G\in\mathbf{G}_d$, we can write
\begin{equation}
    J_{G}=\frac{f_{d}^{(G)}(X_1,\dots,X_d)}{\prod_{i<j}(X_i-X_j)^{2\cdot k_G}}
\end{equation}
in its simplest form, where $f_{d}^{(G)}(X_1,\dots,X_d)\in\mathbb{Z}[X_1,\dots,X_d]$ is a symmetric polynomial. Write $A_{i}=(-1)^{d-i}\sum_{1\leq k_1<\cdots<k_{d-i}\leq d} X_{k_1}\cdots X_{k_{d-i}}$ for $i=0,\dots,d-1$. Since $f_{d}^{(G)}(X_1,\dots,X_d)$ is symmetric, we can write it in terms of $A_0,\dots,A_{d-1}$ as 
\begin{equation}
    f_d^{(G)}(X_1,\dots,X_d)=g_d^{(G)}(A_0,\dots,A_{d-1}).
\end{equation}
Write $\prod_{i<j}(X_i-X_j)^{2}=\Delta(A_0,\dots,A_{d-1})$. Define $\mathcal{F}_d=\{\Delta\}\cup\{g_{d}^{(G)}:G\in\mathbf{G}_d\}\subset\mathbb{Z}[A_0,\dots,A_{d-1}]$ and define $t_d=\#\mathcal{F}_d$, noting that $\mathcal{F}_d$ is a finite set because $\mathbf{G}_d$ is a finite set.
\end{definition}

\begin{theorem}\label{polythm}
Let $\mathcal{F}_d=\{g_{d}^{(1)},\dots,g_{d}^{(t_d)}\}$. The valuations $\textup{ord}(g_{d}^{(i)}(a_0,\dots,a_{d-1}))$ for $i=1,\dots,t_d$ uniquely determine the cluster picture of the separable polynomial $f(x)=c_f(x^d+a_{d-1}x^{d-1}+\cdots+a_0)$ over any discretely valued field $K$.
\end{theorem}

\begin{proof}
By construction, $J_{G,f}=g_d^{(G)}(a_0,\dots,a_{d-1})/\Delta^{k_G}$. 
Thus, since knowing the valuation of $\Delta(a_0,\dots,a_{d-1})$ and $g_{d}^{(i)}(a_0,\dots,a_{d-1})$ for $i=1,\dots,t_d$ means the valuation of $J_{G,f}$ for all $G\in\mathbf{G}_d$ can be calculated, by Theorem \ref{firstalgorithm}, these valuations determine the cluster picture of $f(x)$ over $K$ with depths. 
\end{proof}

\begin{remark}
We can enumerate $\mathbf{G}_d$ for small $d$ to find that $\#\mathcal{F}_3=3$, $\#\mathcal{F}_4=11$ and $\#\mathcal{F}_5=35$, indicating that size of $\mathcal{F}_d$ grows rapidly with $d$, however we do not have a closed or asymptotic formula. 
\end{remark}

As a corollary to Theorem \ref{polythm}, we obtain the following result.

\begin{corollary}
Let $f(x)$ be a separable polynomial over a discretely valued field $K$. The valuation of all polynomials in the coefficients of $f(x)/c_f$ up to degree $\frac{d^2}{4}(d-1)^2$ uniquely determine the cluster picture of $f(x)$ over $K$.
\end{corollary} 

\begin{proof}
Let $G=(V,E,w)\in \mathbf{G}_d$ be a weighted graph on $d$ vertices and let $\textup{num}(J_{G})$ denote the numerator of $J_{G}$ written as a rational function in the variables $X_1,\dots,X_d$. We claim that $\deg(\textup{num}(J_{G}))\leq \frac{d^2}{4}(d-1)^2$. Indeed, there are $\left(\frac{d}{2}\right)$ edges in $G$, and so the maximum weight of an edge in $G$ is $\left(\frac{d}{2}\right)=\frac{d}{2}(d-1)$. There are $\left(\frac{d}{2}\right)=\frac{d}{2}(d-1)$ pairs of variables $(X_i-X_j)^2$ on the denominator of $S_{G}^{\sigma}$, and so when put over a common denominator the numerator has degree less than $\frac{d^2}{4}(d-1)^2$. Hence, the valuation of all symmetric polynomials in the roots of $f(x)$ up to degree $\frac{d^2}{4}(d-1)^2$ uniquely determine the cluster picture of $f(x)$ over $K$. Since the degree of a symmetric polynomial in the roots of $f(x)$ is strictly larger than the degree of the polynomial written in terms of the coefficients, this gives us the result.
\end{proof}

\begin{remark}
It is believed by the author that the process for recovering the cluster picture described in this paper is minimal in the sense that for a discretely valued field $K$ and a weighted graph $G\in \mathbf{G}_d$ with maximum weight $n$, there exists a polynomial $f(x)$ defined over $K$ with $G_n(f)=G$. However, there will be instances where it is possible to extract valuations of the rational functions from previous ones that have been already been calculated.
\end{remark}

\section{Degree $5$ algorithm description}\label{degree5section}

In this section we give a table explicitly describing the rational functions needed to recover the cluster picture of a separable degree $5$ polynomial $f(x)$ over a discretely valued field $K$, and we write out the algorithm by which the cluster picture can be recovered from these rational functions. The algorithm has been implemented for degree $5$ polynomials and is available in the ancillary files to \cite{arxivcode}, along with the rational functions written in terms of the coefficients. When used to calculate the cluster picture of all separable degree $5$ polynomials with coefficients in $\{1,2,3,4,5\}$ over $\mathbb{Q}_7$ it took $2.45$ seconds, whereas the currently implemented method using the SageMath cluster pictures package \cite{alexbest} took $392.53$ seconds. It also appears the SageMath cluster pictures package cannot calculate the cluster picture of a polynomial that has a non-trivial wild inertia action on the roots. 

The auxiliary graphs in $\mathbf{G}_5$ that determine the cluster picture are listed in Table \ref{degree5table} (some are omitted when the cluster picture is uniquely determined by the penultimate auxiliary graph, see Remark \ref{omittingfinal}). Under the column `Summand of $J_{G,f}$', we give a summand of the rational function associated to the auxiliary graph in that row (see Definition \ref{invariants} on how to extract the full rational function from the summand). We write the summand in terms of the roots instead of the rational function in terms of the coefficients of $f(x)$ (see Definition \ref{listofpolys}) because when written in terms of the coefficients they contain too many terms to fit in the paper. When the cluster picture is uniquely determined by the auxiliary graph, we give the cluster picture of a polynomial with such an auxiliary graph in the column `Cluster picture', and we give the depths of the clusters in terms of the rational functions associated to the auxiliary graphs in the column `Depths'. In the column `Example $A_n(G,f)$' we give the value of $A_n(G,f)$ for $G$ the graph in that row and $f(x)=x^5 - 8 x^4 - 823538 x^3 + 4941204 x^2 + 6588464 x + 52706688$ over $\mathbb{Q}_7$ as in Example \ref{degree5ex} below, and we highlight the values that Algorithm \ref{degree5thm} `picks out' to indicate the auxiliary graphs. The auxiliary graphs in the table were enumerated by studying the possible cluster pictures for a degree $5$ polynomial written in \cite{semistable} p.69-70, and considering the possible orderings on the depths of the clusters.

\begin{notation}
In the auxiliary graphs in Table \ref{degree5table}, the edges that do not have a labeled weight have weight $1$. We use $v$ to denote the valuation with respect to a uniformiser of $K$, where $K$ is the base field, and we denote by $\Delta$ the discriminant of $\frac{f(x)}{c_f}$, where $c_f$ is the leading coefficient of $f(x)$. We denote by $x_1,\dots,x_5$ the roots of $f(x)$.
\end{notation}

\begin{algorithm}\label{degree5thm}
For $f(x)$ a degree $5$ polynomial over a discretely valued field $K$ with valuation $v$, the cluster picture of $f(x)$ over $K$ is uniquely determined by calculating the valuations of the rational functions in Table \ref{degree5table} by the following procedure. For a weighted graph $G$ in the table,
\begin{equation}
    J_{G,f}= \sum_{\sigma\in S_5/\textup{Stab}(S_{G,f})}S_{G,f}^\sigma,
\end{equation}
where $S_{G,f}$ is shown in the `Summand of $J_{G,f}$' column, $S_{G,f}^\sigma$ is $S_{G,f}$ under the action of $\sigma\in S_5$ on the roots $x_1,\dots,x_5$ and $\textup{Stab}(S_{G,f})$ is the stabiliser of $S_{G,f}$ under this action, as in Definition \ref{invariants}.

\begin{enumerate}[1.]
        \item \begin{enumerate}[(i)]
            \item Evaluate the value of 
        \begin{equation}
            A_1(G,f)=-\frac{v(J_{G,f})}{2\cdot |E_G|}
        \end{equation}
         for each graph $G$ in Table \ref{degree5table} labelled with one letter. Choose the graph labelled with one letter that has the greatest number of edges out of those that maximise the value of $A_1(-,f)$ and call this $X_1$. The greatest depth in the cluster picture is $d_1(f)=A_1(X_1,f)$. 
         \item If the `Cluster picture' column associated to $X_1$ is not empty, this contains the cluster picture of $f(x)$ over $K$ and the depths of the clusters are written in the `Depths' column in terms of $v(J_{X_1,f})$ and $v(\Delta)$. If the `Cluster picture' column is empty, calculate $X_1X_2$ in Step $2$ below. 
        \end{enumerate} 
        \item \begin{enumerate}[(i)]
            \item Evaluate the value of 
        \begin{equation}
            A_2(X_1G,f)=\frac{-v(J_{X_1G,f})+2v(J_{X_1,f})}{2\cdot (|E_{X_1G}|-|E_{X_1}|)},
        \end{equation}
        for each graph $X_1G$ in Table \ref{degree5table} labelled with two letters and with $X_1$ as the first character. Choose the graph labelled with two letters and with $X_1$ as the first letter that has the greatest number of edges out of those that maximise the value of $A_2(-,f)$ and call this $X_1 X_2$. The second greatest depth in the cluster picture is $d_2(f)=A_2(X_1X_2,f)$. 
        \item If the `Cluster picture' column associated to $X_1X_2$ is not empty, this contains the cluster picture of $f(x)$ over $K$ and the depths of the clusters are written in the `Depths' column in terms of $v(J_{X_1,f})$, $v(J_{X_1X_2,f})$ and $v(\Delta)$. If the `Cluster picture' column is empty, calculate $X_1X_2X_3$ in Step $3$ below.
        \end{enumerate} 
        \item \begin{enumerate}[(i)]
            \item Evaluate the value of 
        \begin{equation}
            A_3(X_1X_2G,f)=\frac{-v(J_{X_1X_2G,f})+2v(J_{X_1X_2,f})-v(J_{X_1,f})}{2\cdot (|E_{X_1X_2G}|-|E_{X_1X_2}|)},
        \end{equation}
        for each graph $X_1X_2G$ in Table \ref{degree5table} labelled with three letters and with $X_1 X_2$ as the first two letters. Choose the graph labelled with three letters and with $X_1 X_2$ as the first two letters that has the greatest number of edges out of those that maximise the value of $A_3(-,f)$ and call this $X_1 X_2 X_3$. The third greatest depth in the cluster picture is $d_3(f)=A_3(X_1X_2X_3,f)$.
        \item The `Cluster picture' column associated to $X_1X_2X_3$ contains the cluster picture of $f(x)$ over $K$ and the depths of the clusters are written in the `Depths' column in terms of $v(J_{X_1,f})$, $v(J_{X_1X_2,f})$, $v(J_{X_1X_2X_3,f})$ and $v(\Delta)$. 
        \end{enumerate}  
    \end{enumerate}
\end{algorithm}

The above algorithm follows immediately from Theorem \ref{main}. In the notation of the paper, $G_1(f)=X_1$, $G_2(f)=X_1X_2$ and $G_3(f)=X_1X_2X_3$. 

\begin{remark}\label{omittingfinal}
On some occasions the cluster picture structure is uniquely determined by the penultimate auxiliary graph. An example of this can be seen for the auxiliary graph $BAA$ in Table \ref{degree5table}. If such a case is reached when performing the algorithm outlined in Theorem \ref{main} for a polynomial of any degree, it is not necessary to calculate the valuation of an extra invariant to calculate the final depth. To illustrate this, note that for a polynomial with $BAA$ as its $3$-rd auxiliary graph, the final auxiliary graph $G_4(f)$ is $BAA$ but with $1$ added to the weight of all preexisting edges and weight $1$ edges between vertices that did not have an edge in $BAA$. The associated rational function is $\frac{1}{\Delta}\cdot J_{BAA}$. This means that $d_4(f)$ can be calculated using the valuation of $\Delta$ and $J_{BAA,f}$, which will have already been calculated at this point in the algorithm.
\end{remark}

\begin{example}\label{degree5ex}
In Table \ref{degree5table} below, we have added an extra column showing the values of $A_n(G,f)$ associated to each graph $G\in \mathbf{G}_5$ for the polynomial 
\begin{equation}
    f(x)=x^5 - 8 x^4 - 823538 x^3 + 4941204 x^2 + 6588464 x + 52706688
\end{equation}
over $\mathbb{Q}_7$, which were calculated using SageMath \cite{sage}. Looking at the values of $A_1(G,f)$ for $G\in\{A,B,C,D, \\ E,F\}$, the largest value is $\frac{7}{2}$ and it is associated to the graph $B$, hence $G_1(f)=B$ and $d_1(f)=\frac{7}{2}$. For the values of $A_2(G,f)$ for $G\in\{BA,BB,BC,BD,BE,BF\}$, the largest is $1$ and it is associated to the graph $BD$, hence $G_2(f)=BD$ and $d_2(f)=1$. Similarly for the values of $A_3(G,f)$ for $G\in\{BDA,BDB,BDC\}$, the largest is $\frac{1}{2}$ and it is associated to the graph $BDA$, hence $G_3(f)=BDA$ and $d_3(f)=\frac{1}{2}$. This uniquely determines the cluster picture to be the one in the row associated to $BDA$, and we can use the valuation of the discriminant $\Delta$ of $f(x)$ and the rational functions to find that $d_4(f)=0$, as demonstrated in Column 5. Thus, labelling the relative depths, the cluster picture of $f(x)$ over $\mathbb{Q}_7$ is 
\begin{center}
    \scalebox{1.75}{
    \clusterpicture             
  \Root[D] {1} {first} {r1};
  \Root[D] {} {r1} {r2};
    \Root[D] {2} {r2} {r3};
    \Root[D] {2.5} {r3} {r4};
    \Root[D] {} {r4} {r5};
  \ClusterD c1[\frac{5}{2}] = (r1)(r2);
  \ClusterD c2[1] = (c1)(r3);
  \ClusterD c3[\frac{1}{2}] = (r4)(r5);
  \ClusterD c4[0] = (c2)(c3);
\endclusterpicture}.
\end{center}
\end{example}

\begin{landscape}
\setlength\extrarowheight{2pt}
\begin{center}
{\setlength{\LTleft}{-0.2in}
 \setlength{\LTright}{0in}
\begin{longtable}{|>{\raggedright\arraybackslash}m{0.9cm}|>{\centering\arraybackslash}m{1.7cm}|>{\centering\arraybackslash}m{7cm}|>{\centering\arraybackslash}m{3.7cm}|>{\scriptsize\raggedright\arraybackslash}m{5.6cm}||>{\centering\arraybackslash}m{1.4cm}|} \caption{Degree $5$ algorithm} \label{degree5table} \\
\hline

\multicolumn{1}{|>{\centering\arraybackslash}m{0.9cm}|}{Graph name} & \multicolumn{1}{|>{\centering\arraybackslash}m{1.7cm}|}{Auxiliary graph $G$} & Summand of $J_{G,f}$ & Cluster picture & \multicolumn{1}{|>{\normalsize\centering\arraybackslash}m{5.6cm}||}{Depths} & Example $A_n(G,f)$ \\

\hline

$A$ & \begin{tikzpicture}
				[scale=0.2,auto=left,every node/.style={circle,fill=black!20,scale=0.6}]
                \node[circle,fill=white] (nonode1) at (0,4)  {};
                \node[circle,fill=white] (nonode1) at (0,-3.5)  {};
                 
				\node (n5) at (378:3)  {};
                \node (n2) at (306:3)  {};
                \node (n4) at (234:3)  {};
                \node (n3) at (162:3)  {};
				\node (n1) at (90:3) {};
				
				\draw (n1) -- (n2);
                     \draw (n1) -- (n3);
                     \draw (n1) -- (n4);
				\draw (n1) -- (n5) ;
                    \draw (n2) -- (n3);
                    \draw (n2) -- (n4);
                    \draw (n2) -- (n5);
				\draw (n3) -- (n4);
                    \draw (n3) -- (n5);
				\draw (n4) -- (n5);
			\end{tikzpicture} & $\frac{1}{\prod_{1\leq i, j\leq 5}(x_i-x_j)^2}$ & \scalebox{1.75}{
    \clusterpicture             
  \Root[D] {1} {first} {r1};
  \Root[D] {} {r1} {r2};
    \Root[D] {} {r2} {r3};
    \Root[D] {} {r3} {r4};
    \Root[D] {} {r4} {r5};
  \ClusterD c1[d_1] = (r1)(r2)(r3)(r4)(r5);
\endclusterpicture} & $d_1=-\frac{1}{20}v(J_{A,f})$ & $\frac{3}{5}$ \\

\hline 

$B$ & \begin{tikzpicture}
				[scale=0.2,auto=left,every node/.style={circle,fill=black!20,scale=0.6}]
                \node[circle,fill=white] (nonode1) at (0,4)  {};
                \node[circle,fill=white] (nonode1) at (0,-3.5)  {};

				\node (n5) at (378:3)  {};
                \node (n2) at (306:3)  {};
                \node (n4) at (234:3)  {};
                \node (n3) at (162:3)  {};
				\node (n1) at (90:3) {};
				
				\draw (n4) -- (n3);
    
			\end{tikzpicture} & $\frac{1}{(x_1-x_2)^2}$ & & & \cellcolor{pink}$ \frac{7}{2}$ \\
   \hline 
$BA$ &     \begin{tikzpicture}
				[scale=0.2,auto=left,every node/.style={circle,fill=black!20,scale=0.6}]
                \node[circle,fill=white] (nonode1) at (0,4)  {};
                \node[circle,fill=white] (nonode1) at (0,-3.5)  {};

				\node (n5) at (378:3)  {};
                \node (n2) at (306:3)  {};
                \node (n4) at (234:3)  {};
                \node (n3) at (162:3)  {};
				\node (n1) at (90:3) {};

				\draw (n2) -- (n5);
				\draw (n3) -- (n4) node [text=black,pos=0.5, left,fill=none] {$2$};
			\end{tikzpicture} & $\frac{1}{(x_1-x_2)^4(x_3-x_4)^2}$ & & & $\frac{1}{2}$ \\

   \hline 
$BAA$ &   \begin{tikzpicture}
				[scale=0.2,auto=left,every node/.style={circle,fill=black!20,scale=0.6}]
                \node[circle,fill=white] (nonode1) at (0,4)  {};
                \node[circle,fill=white] (nonode1) at (0,-3.5)  {};

				\node (n5) at (378:3)  {};
                \node (n2) at (306:3)  {};
                \node (n4) at (234:3)  {};
                \node (n3) at (162:3)  {};
				\node (n1) at (90:3) {};

    \draw (n2) -- (n3);
    \draw (n2) -- (n4);
    \draw (n5) -- (n3);
    \draw (n5) -- (n4);
    
				\draw (n2) -- (n5) node [text=black,pos=0.5, right,fill=none] {$2$};
				\draw (n3) -- (n4) node [text=black,pos=0.5, left,fill=none] {$3$};
			\end{tikzpicture} & $\frac{1}{(x_1-x_2)^6(x_3-x_4)^4(x_1-x_3)^2(x_1-x_4)^2(x_2-x_3)^2(x_2-x_4)^2}$ & \scalebox{1.75}{
    \clusterpicture             
  \Root[D] {1} {first} {r1};
  \Root[D] {} {r1} {r2};
    \Root[D] {3} {r2} {r3};
    \Root[D] {} {r3} {r4};
    \Root[D] {3.5} {r4} {r5};
  \ClusterD c1[d_1] = (r1)(r2);
  \ClusterD c2[d_2] = (r3)(r4);
  \ClusterD c3[d_3] = (c1)(c2);
  \ClusterD c4[d_4] = (c3)(r5);
\endclusterpicture}

$d_1>d_2>d_3>d_4$ & $d_1=-\frac{1}{2}v(J_{B,f})$

$d_2=\frac{1}{2}(-v(J_{BA,f})+2v(J_{B,f}))$

$d_3=\frac{1}{8}(-v(J_{BAA,f})+2v(J_{BA,f})-v(J_{B,f}))$

$d_4 =\frac{1}{8}(v(\Delta)+ v(J_{BAA,f})-v(J_{BA,f}))$

& \\
   \hline 
$BAB$ &    \begin{tikzpicture}
				[scale=0.2,auto=left,every node/.style={circle,fill=black!20,scale=0.6}]
                \node[circle,fill=white] (nonode1) at (0,4)  {};
                \node[circle,fill=white] (nonode1) at (0,-3.5)  {};

				\node (n5) at (378:3)  {};
                \node (n2) at (306:3)  {};
                \node (n4) at (234:3)  {};
                \node (n3) at (162:3)  {};
                
				\node (n1) at (90:3) {};

   \draw (n1) -- (n2);
                     \draw (n1) -- (n3);
                     \draw (n1) -- (n4);
				\draw (n1) -- (n5) ;
                    \draw (n2) -- (n3);
                    \draw (n2) -- (n4);
                    \draw (n3) -- (n5);
				\draw (n4) -- (n5);
    
				\draw (n2) -- (n5) node [text=black,pos=0.5, right,fill=none] {$2$};
				\draw (n3) -- (n4) node [text=black,pos=0.5, left,fill=none] {$3$};
			\end{tikzpicture} & $\frac{1}{(x_1-x_2)^4(x_3-x_4)^2\prod_{1\leq i, j\leq 5} (x_i-x_j)^2}$ & \scalebox{1.75}{
    \clusterpicture             
  \Root[D] {1} {first} {r1};
  \Root[D] {} {r1} {r2};
    \Root[D] {3} {r2} {r3};
    \Root[D] {} {r3} {r4};
    \Root[D] {2.5} {r4} {r5};
  \ClusterD c1[d_1] = (r1)(r2);
  \ClusterD c2[d_2] = (r3)(r4);
  \ClusterD c4[d_3] = (c1)(c2)(r5);
\endclusterpicture} 

$d_1>d_2>d_3$ & $d_1=-\frac{1}{2}v(J_{B,f})$

$d_2=\frac{1}{2}(-v(J_{BA,f})+2v(J_{B,f}))$

$d_3=\frac{1}{16}(-v(J_{BAB,f})+2v(J_{BA,f})-v(J_{B,f}))$

& \\

\hline 
$BAC$ &    \begin{tikzpicture}
				[scale=0.2,auto=left,every node/.style={circle,fill=black!20,scale=0.6}]
                \node[circle,fill=white] (nonode1) at (0,4)  {};
                \node[circle,fill=white] (nonode1) at (0,-3.5)  {};

				\node (n5) at (378:3)  {};
                \node (n2) at (306:3)  {};
                \node (n4) at (234:3)  {};
                \node (n3) at (162:3)  {};
                
				\node (n1) at (90:3) {};

   \draw (n1) -- (n2);
                     
                                    \draw (n1) -- (n5) ;
                
				\draw (n2) -- (n5) node [text=black,pos=0.5, right,fill=none] {$2$};
				\draw (n3) -- (n4) node [text=black,pos=0.5, left,fill=none] {$3$};
			\end{tikzpicture} & $\frac{1}{(x_1-x_2)^6(x_3-x_4)^4(x_5-x_3)^2(x_5-x_4)^2}$ & \scalebox{1.75}{
    \clusterpicture             
  \Root[D] {1} {first} {r1};
  \Root[D] {} {r1} {r2};
    \Root[D] {2} {r2} {r3};
    \Root[D] {2.5} {r3} {r4};
    \Root[D] {} {r4} {r5};
  \ClusterD c1[d_2] = (r1)(r2);
  \ClusterD c2[d_3] = (c1)(r3);
  \ClusterD c3[d_1] = (r4)(r5);
  \ClusterD c4[d_4] = (c2)(c3);
\endclusterpicture} 

$d_1>d_2>d_3>d_4$ &

$d_1=-\frac{1}{2}v(J_{B,f})$

$d_2=\frac{1}{2}(-v(J_{BA,f})+2v(J_{B,f}))$

$d_3 =\frac{1}{4}(-v(J_{B,f}) + 2 v(J_{BA,f}) - v(J_{BAC,f}))$

$d_4=\frac{1}{12}(-v(J_{BA,f}) + v(J_{BAC,f})+ v(\Delta))$

& \\
\hline 
$BAD$ &    \begin{tikzpicture}
				[scale=0.2,auto=left,every node/.style={circle,fill=black!20,scale=0.6}]
                \node[circle,fill=white] (nonode1) at (0,4)  {};
                \node[circle,fill=white] (nonode1) at (0,-3.5)  {};

				\node (n5) at (378:3)  {};
                \node (n2) at (306:3)  {};
                \node (n4) at (234:3)  {};
                \node (n3) at (162:3)  {};
                
				\node (n1) at (90:3) {};

                     \draw (n1) -- (n3);
                     \draw (n1) -- (n4);
				
				\draw (n2) -- (n5) node [text=black,pos=0.5, right,fill=none] {$2$};
				\draw (n3) -- (n4) node [text=black,pos=0.5, left,fill=none] {$3$};
			\end{tikzpicture} & $\frac{1}{(x_1-x_2)^6(x_3-x_4)^4(x_5-x_1)^2(x_5-x_2)^2}$ & \scalebox{1.75}{
    \clusterpicture             
  \Root[D] {1} {first} {r1};
  \Root[D] {} {r1} {r2};
    \Root[D] {2} {r2} {r3};
    \Root[D] {2.5} {r3} {r4};
    \Root[D] {} {r4} {r5};
  \ClusterD c1[d_1] = (r1)(r2);
  \ClusterD c2[d_3] = (c1)(r3);
  \ClusterD c3[d_2] = (r4)(r5);
  \ClusterD c4[d_4] = (c2)(c3);
\endclusterpicture} 

$d_1>d_2>d_3>d_4$ & $d_1=-\frac{1}{2}v(J_{B,f})$

$d_2=\frac{1}{2}(-v(J_{BA,f})+2v(J_{B,f}))$

$d_3 =\frac{1}{4}(-v(J_{B,f}) + 2 v(J_{BA,f}) - v(J_{BAD,f}))$

$d_4=\frac{1}{12}(-v(J_{BA,f}) + v(J_{BAD,f})+ v(\Delta))$

& \\
   \hline 
$BB$ &    \begin{tikzpicture}
				[scale=0.2,auto=left,every node/.style={circle,fill=black!20,scale=0.6}]
                \node[circle,fill=white] (nonode1) at (0,4)  {};
                \node[circle,fill=white] (nonode1) at (0,-3.5)  {};

				\node (n5) at (378:3)  {};
                \node (n2) at (306:3)  {};
                \node (n4) at (234:3)  {};
                \node (n3) at (162:3)  {};
				\node (n1) at (90:3) {};
				
				\draw (n1) -- (n2);
				\draw (n1) -- (n5) ;
				\draw (n2) -- (n5);
				\draw (n3) -- (n4) node [text=black,pos=0.5, left,fill=none] {$2$};
			\end{tikzpicture} & $\frac{1}{(x_1-x_2)^4(x_3-x_4)^2(x_3-x_5)^2(x_4-x_5)^2}$ & \scalebox{1.75}{
    \clusterpicture             
  \Root[D] {1} {first} {r1};
  \Root[D] {} {r1} {r2};
    \Root[D] {3} {r2} {r3};
    \Root[D] {} {r3} {r4};
    \Root[D] {} {r4} {r5};
  \ClusterD c1[d_1] = (r1)(r2);
  \ClusterD c2[d_2] = (r3)(r4)(r5);
  \ClusterD c4[d_3] = (c1)(c2);
\endclusterpicture} 

 $d_1>d_2>d_3$ & $d_1=-\frac{1}{2}v(J_{B,f})$

$d_2=\frac{1}{6}(-v(J_{BB,f})+2v(J_{B,f}))$

$d_3=\frac{1}{12}(v(\Delta)+v(J_{BB,f})-v(J_{B,f}))$

& $\frac{1}{6}$ \\
   \hline 
$BC$ &    \begin{tikzpicture}
				[scale=0.2,auto=left,every node/.style={circle,fill=black!20,scale=0.6}]
                \node[circle,fill=white] (nonode1) at (0,4)  {};
                \node[circle,fill=white] (nonode1) at (0,-3.5)  {};

				\node (n5) at (378:3)  {};
                \node (n2) at (306:3)  {};
                \node (n4) at (234:3)  {};
                \node (n3) at (162:3)  {};
				\node (n1) at (90:3) {};
				
				\draw (n1) -- (n2);
                     \draw (n1) -- (n3);
                     \draw (n1) -- (n4);
				\draw (n1) -- (n5) ;
                    \draw (n2) -- (n3);
                    \draw (n2) -- (n4);
                    \draw (n2) -- (n5);
                    \draw (n3) -- (n5);
				\draw (n4) -- (n5);
				\draw (n3) -- (n4) node [text=black,pos=0.5, left,fill=none] {$2$};
			\end{tikzpicture} & $\frac{1}{(x_1-x_2)^2\prod_{1\leq i, j \leq 5} (x_i-x_j)^2}$ & \scalebox{1.75}{
    \clusterpicture             
  \Root[D] {1} {first} {r1};
  \Root[D] {} {r1} {r2};
    \Root[D] {3} {r2} {r3};
    \Root[D] {} {r3} {r4};
    \Root[D] {} {r4} {r5};
  \ClusterD c1[d_1] = (r1)(r2);
  \ClusterD c2[d_2] = (c1)(r3)(r4)(r5);
\endclusterpicture} 

$d_1>d_2$ & $d_1=-\frac{1}{2}v(J_{B,f})$

$d_2=\frac{1}{18}(-v(J_{BC,f})+2v(J_{B,f}))$

& $\frac{5}{18}$ \\
\hline 
$BD$ &    \begin{tikzpicture}
				[scale=0.2,auto=left,every node/.style={circle,fill=black!20,scale=0.6}]
                \node[circle,fill=white] (nonode1) at (0,4)  {};
                \node[circle,fill=white] (nonode1) at (0,-3.5)  {};
                
				\node (n5) at (378:3)  {};
                \node (n2) at (306:3)  {};
                \node (n4) at (234:3)  {};
                \node (n3) at (162:3)  {};
				\node (n1) at (90:3) {};

                    \draw (n1) -- (n3);
                    \draw (n1) -- (n4);
				\draw (n3) -- (n4) node [text=black,pos=0.5, left,fill=none] {$2$};
			\end{tikzpicture} & $\frac{1}{(x_1-x_2)^4(x_1-x_3)^2(x_2-x_3)^2}$ & & & \cellcolor{pink} $1$ \\
   \hline 
$BDA$ &    \begin{tikzpicture}
				[scale=0.2,auto=left,every node/.style={circle,fill=black!20,scale=0.6}]
                \node[circle,fill=white] (nonode1) at (0,4)  {};
                \node[circle,fill=white] (nonode1) at (0,-3.5)  {};
                 \node[circle,fill=white] (nonode1) at (4,0)  {};
                
				\node (n5) at (378:3)  {};
                \node (n2) at (306:3)  {};
                \node (n4) at (234:3)  {};
                \node (n3) at (162:3)  {};
				\node (n1) at (90:3) {};
				
		          \draw (n5) -- (n2);
                    
                    \draw (n1) -- (n3) node [text=black,pos=0.5, left,fill=none] {$2$};
                    \draw (n1) -- (n4) node [text=black,pos=0.5, right,fill=none] {$2$};
				\draw (n3) -- (n4) node [text=black,pos=0.5, left,fill=none] {$3$};
			\end{tikzpicture} & $\frac{1}{(x_1-x_2)^6(x_1-x_3)^4(x_2-x_3)^4(x_4-x_5)^2}$ & \scalebox{1.75}{
    \clusterpicture             
  \Root[D] {1} {first} {r1};
  \Root[D] {} {r1} {r2};
    \Root[D] {2} {r2} {r3};
    \Root[D] {2.5} {r3} {r4};
    \Root[D] {} {r4} {r5};
  \ClusterD c1[d_1] = (r1)(r2);
  \ClusterD c2[d_2] = (c1)(r3);
  \ClusterD c3[d_3] = (r4)(r5);
  \ClusterD c4[d_4] = (c2)(c3);
\endclusterpicture} 

$d_1>d_2>d_3>d_4$ & $d_1=-\frac{1}{2}v(J_{B,f})$

$d_2=\frac{1}{4}(-v(J_{BD,f})+2v(J_{B,f}))$

$d_3= \frac{1}{2}(-v(J_{BDA,f})+2v(J_{BD,f})-v(J_{B,f}))$

$d_4=\frac{1}{12}(v(\Delta)+v(J_{BDA,f})-v(J_{BD,f}))$

& \cellcolor{pink} $\frac{1}{2}$ \\
\hline 
 $BDB$ &   \begin{tikzpicture}
				[scale=0.2,auto=left,every node/.style={circle,fill=black!20,scale=0.6}]
                \node[circle,fill=white] (nonode1) at (0,4)  {};
                \node[circle,fill=white] (nonode1) at (0,-3.5)  {};
                 \node[circle,fill=white] (nonode1) at (4,0)  {};
                
				\node (n5) at (378:3)  {};
                \node (n2) at (306:3)  {};
                \node (n4) at (234:3)  {};
                \node (n3) at (162:3)  {};
				\node (n1) at (90:3) {};
				
		\draw (n1) -- (n5);
  \draw (n3) -- (n5);
  \draw (n4) -- (n5);
  
                    \draw (n1) -- (n3) node [text=black,pos=0.5, left,fill=none] {$2$};
                    \draw (n1) -- (n4) node [text=black,pos=0.9, above,fill=none] {$2$};
				\draw (n3) -- (n4) node [text=black,pos=0.5, left,fill=none] {$3$};
			\end{tikzpicture} & $\frac{1}{(x_1-x_2)^6(x_1-x_3)^4(x_2-x_3)^4(x_1-x_4)^2(x_2-x_4)^2(x_3-x_4)^2}$ & \scalebox{1.75}{
    \clusterpicture             
  \Root[D] {1} {first} {r1};
  \Root[D] {} {r1} {r2};
    \Root[D] {2} {r2} {r3};
    \Root[D] {2} {r3} {r4};
    \Root[D] {2} {r4} {r5};
  \ClusterD c1[d_1] = (r1)(r2);
  \ClusterD c2[d_2] = (c1)(r3);
  \ClusterD c3[d_3] = (c2)(r4);
  \ClusterD c4[d_4] = (c3)(r5);
\endclusterpicture} 

$d_1>d_2>d_3>d_4$ & $d_1=-\frac{1}{2}v(J_{B,f})$ 

$d_2=\frac{1}{4}(-v(J_{BD,f})+2v(J_{B,f}))$

$d_3=\frac{1}{6}(-v(J_{BDB,f})+2v(J_{BD,f})-v(J_{B,f}))$

$d_4=\frac{1}{8}(v(\Delta) +v(J_{BDB,f})-v(J_{BD,f}))$

& $0$ \\
   \hline 
$BDC$ &    \begin{tikzpicture}
				[scale=0.2,auto=left,every node/.style={circle,fill=black!20,scale=0.6}]
                \node[circle,fill=white] (nonode1) at (0,4)  {};
                \node[circle,fill=white] (nonode1) at (0,-3.5)  {};
                 \node[circle,fill=white] (nonode1) at (4,0)  {};
                
				\node (n5) at (378:3)  {};
                \node (n2) at (306:3)  {};
                \node (n4) at (234:3)  {};
                \node (n3) at (162:3)  {};
				\node (n1) at (90:3) {};

                    \draw (n1) -- (n2);
				\draw (n1) -- (n5) ;
                    \draw (n2) -- (n3);
                    \draw (n2) -- (n4);
                    \draw (n2) -- (n5);
                    \draw (n3) -- (n5);
				\draw (n4) -- (n5);
		
                    \draw (n1) -- (n3) node [text=black,pos=0.5, left,fill=none] {$2$};
                    \draw (n1) -- (n4) node [text=black,pos=0.35, above,fill=none] {$2$};
				\draw (n3) -- (n4) node [text=black,pos=0.5, left,fill=none] {$3$};
			\end{tikzpicture} & $\frac{1}{(x_1-x_2)^4(x_1-x_3)^2(x_2-x_3)^2\prod_{1\leq i, j \leq 5} (x_i-x_j)^2}$ & \scalebox{1.75}{
    \clusterpicture             
  \Root[D] {1} {first} {r1};
  \Root[D] {} {r1} {r2};
    \Root[D] {2} {r2} {r3};
    \Root[D] {2} {r3} {r4};
    \Root[D] {} {r4} {r5};
  \ClusterD c1[d_1] = (r1)(r2);
  \ClusterD c2[d_2] = (c1)(r3);
  \ClusterD c4[d_3] = (c2)(r4)(r5);
\endclusterpicture} 

$d_1>d_2>d_3$ & $d_1=-\frac{1}{2}v(J_{B,f})$

$d_2=\frac{1}{4}(-v(J_{BD,f})+2v(J_{B,f}))$

$d_3=\frac{1}{14}(-v(J_{BDC,f})+2v(J_{BD,f})-v(J_{B,f}))$

& $0$ \\
   \hline 
$BE$ &    \begin{tikzpicture}
				[scale=0.2,auto=left,every node/.style={circle,fill=black!20,scale=0.6}]
                \node[circle,fill=white] (nonode1) at (0,4)  {};
                \node[circle,fill=white] (nonode1) at (0,-3.5)  {};
                
				\node (n5) at (378:3)  {};
                \node (n2) at (306:3)  {};
                \node (n4) at (234:3)  {};
                \node (n3) at (162:3)  {};
				\node (n1) at (90:3) {};
				
				\draw (n3) -- (n2);
				\draw (n3) -- (n5) ;
    
				\draw (n2) -- (n5);
                    \draw (n2) -- (n4);
                    \draw (n4) -- (n5);
				\draw (n3) -- (n4) node [text=black,pos=0.5, left,fill=none] {$2$};
			\end{tikzpicture} & $\frac{1}{(x_1-x_2)^4(x_1-x_3)^2(x_1-x_4)^2(x_2-x_3)^2(x_2-x_4)^2(x_3-x_4)^2}$ & \scalebox{1.75}{
    \clusterpicture             
  \Root[D] {1} {first} {r1};
  \Root[D] {} {r1} {r2};
    \Root[D] {2} {r2} {r3};
    \Root[D] {} {r3} {r4};
    \Root[D] {2} {r4} {r5};
  \ClusterD c1[d_1] = (r1)(r2);
  \ClusterD c2[d_2] = (c1)(r3)(r4);
  \ClusterD c4[d_3] = (c2)(r5);
\endclusterpicture} 

$d_1>d_2>d_3$ & $d_1=-\frac{1}{2}v(J_{B,f})$ 

$d_2=\frac{1}{10}(-v(J_{BE,f})+2v(J_{B,f}))$

$d_3=\frac{1}{8}(v(\Delta)+v(J_{BE,f})-v(J_{B,f}))$

& $\frac{2}{5}$ \\
   \hline 
$BF$ &    \begin{tikzpicture}
				[scale=0.2,auto=left,every node/.style={circle,fill=black!20,scale=0.6}]
                \node[circle,fill=white] (nonode1) at (0,4)  {};
                \node[circle,fill=white] (nonode1) at (0,-3.5)  {};

				\node (n5) at (378:3)  {};
                \node (n2) at (306:3)  {};
                \node (n4) at (234:3)  {};
                \node (n3) at (162:3)  {};
				\node (n1) at (90:3) {};

                    \draw (n2) -- (n5);
				 \draw (n1) -- (n3);
                    \draw (n1) -- (n4);
				\draw (n3) -- (n4) node [text=black,pos=0.5, left,fill=none] {$2$};
			\end{tikzpicture}  & $\frac{1}{(x_1-x_2)^4(x_1-x_3)^2(x_2-x_3)^2(x_4-x_5)^2}$ & \scalebox{1.75}{
    \clusterpicture             
  \Root[D] {1} {first} {r1};
  \Root[D] {} {r1} {r2};
    \Root[D] {2} {r2} {r3};
    \Root[D] {2} {r3} {r4};
    \Root[D] {} {r4} {r5};
  \ClusterD c1[d_1] = (r1)(r2);
  \ClusterD c2[d_2] = (c1)(r3);
   \ClusterD c3[d_2] = (r4)(r5);
  \ClusterD c4[d_3] = (c2)(c3);
\endclusterpicture} 

$d_1>d_2>d_3$ & $d_1=-\frac{1}{2}v(J_{B,f})$ 

$d_2=\frac{1}{6}(-v(J_{BF,f})+2v(J_{B,f}))$ 

$d_3=\frac{1}{12}(v(\Delta)+v(J_{BF,f})-v(J_{B,f}))$

& $\frac{5}{6}$ \\
   \hline 

$C$ & \begin{tikzpicture}
				[scale=0.2,auto=left,every node/.style={circle,fill=black!20,scale=0.6}]
                \node[circle,fill=white] (nonode1) at (0,4)  {};
                \node[circle,fill=white] (nonode1) at (0,-3.5)  {};

				\node (n5) at (378:3)  {};
                \node (n2) at (306:3)  {};
                \node (n4) at (234:3)  {};
                \node (n3) at (162:3)  {};
				\node (n1) at (90:3) {};
				
				\draw (n1) -- (n2);
				\draw (n1) -- (n5) ;
				\draw (n2) -- (n5);
			\end{tikzpicture} & $\frac{1}{(x_1-x_2)^2(x_1-x_3)^2(x_2-x_3)^2}$ & & & $\frac{11}{6}$ \\
   \hline 
$CA$ &    \begin{tikzpicture}
				[scale=0.2,auto=left,every node/.style={circle,fill=black!20,scale=0.6}]
                \node[circle,fill=white] (nonode1) at (0,4)  {};
                \node[circle,fill=white] (nonode1) at (0,-3.5)  {};

				\node (n5) at (378:3)  {};
                \node (n2) at (306:3)  {};
                \node (n4) at (234:3)  {};
                \node (n3) at (162:3)  {};
				\node (n1) at (90:3) {};

				\draw (n3) -- (n4);
    
				\draw (n1) -- (n2) node [text=black,pos=0.5, left,fill=none] {$2$};
				\draw (n1) -- (n5) node [text=black,pos=0.5, above,fill=none] {$2$};
				\draw (n2) -- (n5) node [text=black,pos=0.5, right,fill=none] {$2$};
			\end{tikzpicture} & $\frac{1}{(x_1-x_2)^4(x_1-x_3)^4(x_2-x_3)^4(x_4-x_5)^2}$ & \scalebox{1.75}{
    \clusterpicture             
  \Root[D] {1} {first} {r1};
  \Root[D] {} {r1} {r2};
    \Root[D] {} {r2} {r3};
    \Root[D] {3} {r3} {r4};
    \Root[D] {} {r4} {r5};
  \ClusterD c1[d_1] = (r1)(r2)(r3);
  \ClusterD c2[d_2] = (r4)(r5);
   \ClusterD c3[d_3] = (c1)(c2);
\endclusterpicture} 

$d_1>d_2>d_3$ & $d_1=-\frac{1}{6}v(J_{C,f})$ 

$d_2= \frac{1}{2}(-v(J_{CA,f})+2v(J_{C,f}))$

$ d_3= \frac{1}{12}(v(\Delta)+v(J_{CA,f})-v(J_{C,f}))$

& \\
\hline 
 $CB$ &   \begin{tikzpicture}
				[scale=0.2,auto=left,every node/.style={circle,fill=black!20,scale=0.6}]
                \node[circle,fill=white] (nonode1) at (0,4)  {};
                \node[circle,fill=white] (nonode1) at (0,-3.5)  {};

				\node (n5) at (378:3)  {};
                \node (n2) at (306:3)  {};
                \node (n4) at (234:3)  {};
                \node (n3) at (162:3)  {};
				\node (n1) at (90:3) {};

                     \draw (n1) -- (n3);
                     \draw (n1) -- (n4);
			
                    \draw (n2) -- (n3);
                    \draw (n2) -- (n4);
            
				\draw (n3) -- (n4);
                    \draw (n3) -- (n5);
				\draw (n4) -- (n5);
				
				\draw (n1) -- (n2) node [text=black,pos=0.4, above,fill=none] {$2$};
				\draw (n1) -- (n5) node [text=black,pos=0.5, above,fill=none] {$2$};
				\draw (n2) -- (n5) node [text=black,pos=0.5, right,fill=none] {$2$};
			\end{tikzpicture} & $\frac{1}{(x_1-x_2)^2(x_1-x_3)^2(x_2-x_3)^2\prod_{1\leq i,j \leq 5}(x_i-x_j)^2}$ & \scalebox{1.75}{
    \clusterpicture             
  \Root[D] {1} {first} {r1};
  \Root[D] {} {r1} {r2};
    \Root[D] {} {r2} {r3};
    \Root[D] {2.5} {r3} {r4};
    \Root[D] {} {r4} {r5};
  \ClusterD c1[d_1] = (r1)(r2)(r3);
  \ClusterD c2[d_2] = (c1)(r4)(r5);
\endclusterpicture} 

$d_1>d_2$ & $d_1=-\frac{1}{6}v(J_{C,f})$

$d_2=\frac{1}{14}(-v(J_{CB,f})+2v(J_{C,f}))$

& \\

\hline 
   
$CC$ &    \begin{tikzpicture}
				[scale=0.2,auto=left,every node/.style={circle,fill=black!20,scale=0.6}]
                \node[circle,fill=white] (nonode1) at (0,4)  {};
                \node[circle,fill=white] (nonode1) at (0,-3.5)  {};

				\node (n5) at (378:3)  {};
                \node (n2) at (306:3)  {};
                \node (n4) at (234:3)  {};
                \node (n3) at (162:3)  {};
				\node (n1) at (90:3) {};

                    \draw (n1) -- (n4);
                    \draw (n2) -- (n4);
                    \draw (n5) -- (n4);
                    
				\draw (n1) -- (n2) node [text=black,pos=0.5, above,fill=none] {$2$};
				\draw (n1) -- (n5) node [text=black,pos=0.5, above,fill=none] {$2$};
				\draw (n2) -- (n5) node [text=black,pos=0.5, right,fill=none] {$2$};
			\end{tikzpicture} & $\frac{1}{(x_1-x_2)^4(x_1-x_3)^4(x_2-x_3)^4(x_1-x_4)^2(x_2-x_4)^2(x_3-x_4)^2}$ & \scalebox{1.75}{
    \clusterpicture             
  \Root[D] {1} {first} {r1};
  \Root[D] {} {r1} {r2};
    \Root[D] {} {r2} {r3};
    \Root[D] {2} {r3} {r4};
    \Root[D] {2} {r4} {r5};
  \ClusterD c1[d_1] = (r1)(r2)(r3);
  \ClusterD c2[d_2] = (c1)(r4);
   \ClusterD c3[d_3] = (c2)(r5);
\endclusterpicture} 

$d_1>d_2>d_3$ & $d_1=-\frac{1}{6}v(J_{C,f})$

$d_2=\frac{1}{6}(-v(J_{CC,f})+2v(J_{C,f}))$

$d_3=\frac{1}{8}(v(\Delta)+v(J_{CC,f})-v(J_{C,f}))$

& \\

\hline

 $D$ & \begin{tikzpicture}
				[scale=0.2,auto=left,every node/.style={circle,fill=black!20,scale=0.6}]
                \node[circle,fill=white] (nonode1) at (0,4)  {};
                \node[circle,fill=white] (nonode1) at (0,-3.5)  {};

				\node (n5) at (378:3)  {};
                \node (n2) at (306:3)  {};
                \node (n4) at (234:3)  {};
                \node (n3) at (162:3)  {};
				\node (n1) at (90:3) {};
				
				\draw (n3) -- (n2);
				\draw (n3) -- (n5) ;
                    \draw (n3) -- (n4) ;
    
				\draw (n2) -- (n5);
                    \draw (n2) -- (n4);
                    \draw (n4) -- (n5);
			\end{tikzpicture} & $\frac{1}{(x_1-x_2)^2(x_1-x_3)^2(x_1-x_4)^2(x_2-x_3)^2(x_2-x_4)^2(x_3-x_4)^2}$ & \scalebox{1.75}{
    \clusterpicture             
  \Root[D] {1} {first} {r1};
  \Root[D] {} {r1} {r2};
    \Root[D] {} {r2} {r3};
    \Root[D] {} {r3} {r4};
    \Root[D] {2} {r4} {r5};
  \ClusterD c1[d_1] = (r1)(r2)(r3)(r4);
  \ClusterD c2[d_2] = (c1)(r5);
\endclusterpicture} 

$d_1>d_2$ & $d_1=-\frac{1}{12}v(J_{D,f})$

$d_2=\frac{1}{8}(v(\Delta)+v(J_{D,f}))$ & $\frac{11}{12}$ \\

\hline 

$E$ & \begin{tikzpicture}
				[scale=0.2,auto=left,every node/.style={circle,fill=black!20,scale=0.6}]
                \node[circle,fill=white] (nonode1) at (0,4)  {};
                \node[circle,fill=white] (nonode1) at (0,-3.5)  {};

				\node (n5) at (378:3)  {};
                \node (n2) at (306:3)  {};
                \node (n4) at (234:3)  {};
                \node (n3) at (162:3)  {};
				\node (n1) at (90:3) {};
				
				\draw (n4) -- (n3);
                    \draw (n2) -- (n5);
    
			\end{tikzpicture} & $\frac{1}{(x_1-x_2)^2(x_3-x_4)^2}$ & & & $2$ \\
   \hline 
  $EA$ &  \begin{tikzpicture}
				[scale=0.2,auto=left,every node/.style={circle,fill=black!20,scale=0.6}]
                \node[circle,fill=white] (nonode1) at (0,4)  {};
                \node[circle,fill=white] (nonode1) at (0,-3.5)  {};

				\node (n5) at (378:3)  {};
                \node (n2) at (306:3)  {};
                \node (n4) at (234:3)  {};
                \node (n3) at (162:3)  {};
				\node (n1) at (90:3) {};

    \draw (n1) -- (n2);
                     \draw (n1) -- (n3);
                     \draw (n1) -- (n4);
				\draw (n1) -- (n5) ;
                    \draw (n2) -- (n3);
                    \draw (n2) -- (n4);
                    \draw (n3) -- (n5);
				\draw (n4) -- (n5);
				
				\draw (n4) -- (n3) node [text=black,pos=0.4, left,fill=none] {$2$};
                    \draw (n2) -- (n5) node [text=black,pos=0.4, right,fill=none] {$2$};
    
			\end{tikzpicture} & $\frac{1}{(x_1-x_2)^2(x_3-x_4)^2\prod_{1\leq i, j \leq 5}(x_i-x_j)^2}$ & \scalebox{1.75}{
    \clusterpicture             
  \Root[D] {1} {first} {r1};
  \Root[D] {} {r1} {r2};
    \Root[D] {3} {r2} {r3};
    \Root[D] {} {r3} {r4};
    \Root[D] {2} {r4} {r5};
  \ClusterD c1[d_1] = (r1)(r2);
  \ClusterD c2[d_1] = (r3)(r4);
   \ClusterD c3[d_2] = (c1)(c2)(r5);
\endclusterpicture} 

$d_1>d_2$ & $d_1 =-\frac{1}{4}v(J_{E,f})$

$d_2=\frac{1}{16}(-v(J_{EA,f})+2v(J_{E,f}))$ & \\

\hline 
   
$EB$ &    \begin{tikzpicture}
				[scale=0.2,auto=left,every node/.style={circle,fill=black!20,scale=0.6}]
                \node[circle,fill=white] (nonode1) at (0,4)  {};
                \node[circle,fill=white] (nonode1) at (0,-3.5)  {};

				\node (n5) at (378:3)  {};
                \node (n2) at (306:3)  {};
                \node (n4) at (234:3)  {};
                \node (n3) at (162:3)  {};
				\node (n1) at (90:3) {};

    \draw (n3) -- (n2);
				\draw (n3) -- (n5) ;
                    \draw (n2) -- (n4);
                    \draw (n4) -- (n5);
				
				\draw (n4) -- (n3) node [text=black,pos=0.4, left,fill=none] {$2$};
                    \draw (n2) -- (n5) node [text=black,pos=0.4, right,fill=none] {$2$};
    
			\end{tikzpicture} & $\frac{1}{(x_1-x_2)^4(x_3-x_4)^4(x_1-x_3)^2(x_1-x_4)^2(x_2-x_3)^2(x_2-x_4)^2}$ & \scalebox{1.75}{
    \clusterpicture             
  \Root[D] {1} {first} {r1};
  \Root[D] {} {r1} {r2};
    \Root[D] {3} {r2} {r3};
    \Root[D] {} {r3} {r4};
    \Root[D] {3.5} {r4} {r5};
  \ClusterD c1[d_1] = (r1)(r2);
  \ClusterD c2[d_1] = (r3)(r4);
   \ClusterD c3[d_2] = (c1)(c2);
   \ClusterD c4[d_3] = (c3)(r5);
\endclusterpicture} 

$d_1>d_2>d_3$ & $d_1 =-\frac{1}{4}v(J_{E,f})$

$d_2=\frac{1}{8}(-v(J_{EB,f})+2v(J_{E,f}))$

$d_3=\frac{1}{8}(v(\Delta)+v(J_{EB,f})-v(J_{E,f}))$
& \\

\hline 
  
 $EC$ &   \begin{tikzpicture}
				[scale=0.2,auto=left,every node/.style={circle,fill=black!20,scale=0.6}]
                \node[circle,fill=white] (nonode1) at (0,4)  {};
                \node[circle,fill=white] (nonode1) at (0,-3.5)  {};

				\node (n5) at (378:3)  {};
                \node (n2) at (306:3)  {};
                \node (n4) at (234:3)  {};
                \node (n3) at (162:3)  {};
				\node (n1) at (90:3) {};

                    \draw (n4) -- (n1);
                    \draw (n1) -- (n3);
				\draw (n4) -- (n3) node [text=black,pos=0.4, left,fill=none] {$2$};
                    \draw (n2) -- (n5) node [text=black,pos=0.4, right,fill=none] {$2$};
    
			\end{tikzpicture} & $\frac{1}{(x_1-x_2)^4(x_3-x_4)^4(x_1-x_5)^2(x_2-x_5)^2}$ & \scalebox{1.75}{
    \clusterpicture             
  \Root[D] {1} {first} {r1};
  \Root[D] {} {r1} {r2};
    \Root[D] {4} {r2} {r3};
    \Root[D] {} {r3} {r4};
    \Root[D] {2} {r4} {r5};
  \ClusterD c1[d_1] = (r1)(r2);
  \ClusterD c2[d_1] = (r3)(r4);
   \ClusterD c3[d_2] = (c2)(r5);
   \ClusterD c4[d_3] = (c1)(c3);
\endclusterpicture} 

$d_1>d_2>d_3$ & $d_1 =-\frac{1}{4}v(J_{E,f})$

$d_2=\frac{1}{4}(-v(J_{EC,f})+2v(J_{E,f}))$

$d_3=\frac{1}{12}(v(\Delta)+v(J_{EC,f})-v(J_{E,f}))$

& \\
   \hline 

$F$ & \begin{tikzpicture}
				[scale=0.2,auto=left,every node/.style={circle,fill=black!20,scale=0.6}]
                \node[circle,fill=white] (nonode1) at (0,4)  {};
                \node[circle,fill=white] (nonode1) at (0,-3.5)  {};

				\node (n5) at (378:3)  {};
                \node (n2) at (306:3)  {};
                \node (n4) at (234:3)  {};
                \node (n3) at (162:3)  {};
				\node (n1) at (90:3) {};
				
				\draw (n1) -- (n2);
				\draw (n1) -- (n5) ;
				\draw (n2) -- (n5);
				\draw (n3) -- (n4);
			\end{tikzpicture} & $\frac{1}{(x_1-x_2)^2(x_3-x_4)^2(x_3-x_5)^2(x_4-x_5)^2}$ & \scalebox{1.75}{
    \clusterpicture             
  \Root[D] {1} {first} {r1};
  \Root[D] {} {r1} {r2};
    \Root[D] {3} {r2} {r3};
    \Root[D] {} {r3} {r4};
    \Root[D] {} {r4} {r5};
  \ClusterD c1[d_1] = (r1)(r2);
  \ClusterD c2[d_1] = (r3)(r4)(r5);
   \ClusterD c3[d_2] = (c1)(c2);
\endclusterpicture} 

$d_1>d_2$ & $d_1 =-\frac{1}{8}v(J_{F,f})$

$d_2= \frac{1}{12}(v(\Delta)+v(J_{F,f}))$

& $\frac{3}{2}$ \\ 
   \hline 

\end{longtable}}

\end{center}

\end{landscape}

\end{document}